\documentclass{amsart}
\usepackage{amssymb}
\usepackage{amsmath}
\usepackage[noadjust]{cite}
\usepackage{bbm}
\usepackage{breqn}

\title{Worst Unstable Points of a Hilbert Scheme}
\author{Cheolgyu Lee}
\thanks{The author wants to thank his advisor, Donghoon Hyeon. The author was partially supported by the following grants funded by
the government of Korea: NRF grant 2011-0030044 (SRC-GAIA) and
 NRF grant NRF-2013R1A1A2010649.}
\address[CL]{Department of Mathematics, POSTECH, Pohang, Gyungbuk 790-784, R. O. Korea}
\email{ghost279@postech.ac.kr}
\date{\today}

\newtheorem{theorem}{Theorem}[section]
\newtheorem{lemma}[theorem]{Lemma}

\newtheorem{corollary}[theorem]{Corollary}

\newenvironment{definition}[1][Definition]{\begin{trivlist}
\item[\hskip \labelsep {\bfseries #1}]}{\end{trivlist}}
\newenvironment{example}[1][Example]{\begin{trivlist}
\item[\hskip \labelsep {\bfseries #1}]}{\end{trivlist}}

\newcommand{\sslash}{\mathbin{/\mkern-6mu/}}

\begin{document}
\begin{abstract}
In this paper, we describe the worst unstable points of a Hilbert scheme for some special Hilbert polynomials and ambient spaces using Murai's work on Gotzmann monomial sets. We investigate the geometry of the projective schemes represented by worst unstable Hilbert points and see that in certain cases that they fail to be $K$-stable or attain
maximal regularity.
\end{abstract}
\maketitle
\section{Introduction}
Let $k$ be an algebraically closed field and let $S=k[x_{0}, \ldots, x_{r}]$ be a polynomial ring over $k$ where $r\geq 1$. Let $P$ be the Hilbert polynomial of $S/I$ for some homogeneous ideal $I$ of $S$. In this paper, $g_{P}$ is the Gotzmann number associated to $P$ defined in \cite{Gotzmann}. For $d\geq g_{P}$ there are closed immersions
\begin{displaymath}
\textrm{Hilb}^{P}(\mathbb{P}^{r}_{k}) \hookrightarrow \textrm{Gr}(S_{d}, Q_{P}(d))\hookrightarrow \mathbb{P} \bigg(\bigwedge^{Q_{P}(d)}S_{d}\bigg)
\end{displaymath}
which are compatible with the canonical linear action of the general linear group $\textup{GL}_{r+1}(k)$ where
\begin{displaymath}
Q_{P}(d)=\binom{r+d}{r}-P(d).
\end{displaymath}
Consider the GIT quotient $\textrm{Hilb}^{P}(\mathbb{P}^{r}_{k})\sslash_{d}\textup{SL}_{r+1}(k)$ with respect to the above Pl\"ucker embedding corresponding to $d$ and another GIT-quotient $\mathbb{P}(\bigwedge^{Q_{P}(d)}S_{d})\sslash\textup{SL}_{r+1}(k)$. We have the Hesselink stratification of $\mathbb{P}(\bigwedge^{Q_{P}(d)}S_{d})$ described in \cite[p. 9]{Hoskins}. That is, there is a stratification of the unstable locus
\begin{equation}
\label{strata}
\mathbb{P}\bigg(\bigwedge^{b}S_{d}\bigg)^{\textup{us}}=\coprod_{[\lambda], d'}{E_{[\lambda], d'}^{d, b}}
\end{equation}
for all $d, b\in\mathbb{N}$. An unstable point $x$ belongs to a stratum $E^{d, b}_{[\lambda], d'}$ if the conjugacy class $[\lambda]$ contains a 1-parameter subgroup that is adapted to $x$ and the Kempf index \cite{Kempf} of $x$ is $d'$. Setting $b=Q_{P}(d)$ in \eqref{strata}, we obtain the Hesselink stratification
\begin{displaymath}
\textrm{Hilb}^{P}(\mathbb{P}^{r}_{k})^{\textrm{us}}_{d}=\coprod_{[\lambda], d'}{E_{[\lambda], d'}^{d, Q_{P}(d)}}\cap \textrm{Hilb}^{P}(\mathbb{P}^{r}_{k})
\end{displaymath} 
of the Hilbert scheme $\textrm{Hilb}^{P}(\mathbb{P}^{r}_{k})$ with respect to the Pl\"ucker embedding into $\mathbb{P}(\bigwedge^{Q_{P}(d)}S_{d})$. Now we are ready to define worst unstable points of a Hilbert scheme $\textrm{Hilb}^{P}(\mathbb{P}^{r}_{k})$ for an arbitrary choice of $r$ and $P$.
\begin{definition}
For $r, d\in\mathbb{N}$ and a Hilbert polynomial $P$, let $\Gamma(\textup{SL}_{r+1}(k))$ be the group of all 1-parameter subgroups of $\textup{SL}_{r+1}(k)$ and
\begin{displaymath}
\sigma_{\textup{max}}(r, d, P):=\max\big\lbrace \sigma\in\mathbb{R}_{>0}\big\vert\exists\lambda\in\Gamma(\textup{SL}_{r+1}(k))\textup{ s.t. } {E_{[\lambda], \sigma}^{d, Q_{P}(d)}}\cap \textrm{Hilb}^{P}(\mathbb{P}^{r}_{k})\neq \emptyset \big\rbrace.
\end{displaymath}
A point $x\in \textrm{Hilb}^{P}(\mathbb{P}^{r}_{k})$ is a worst unstable point of $\textup{Hilb}^{P}(\mathbb{P}^{r}_{k})$ with respect to $d$ if $x\in E_{[\lambda], \sigma_{\textup{max}}(r, d, P)}^{d, Q_{P}(d)}$ for some 1-parameter subgroup $\lambda$ of $\textup{SL}_{r+1}(k)$. If $x\in \textrm{Hilb}^{P}(\mathbb{P}^{r}_{k})$ is a worst unstable point of $\textup{Hilb}^{P}(\mathbb{P}^{r}_{k})$ with respect to all but finitely many nonnegative integers, then let us call $x$ as a worst unstable point of $\textrm{Hilb}^{P}(\mathbb{P}^{r}_{k})$ or a worst unstable Hilbert point, shortly. Also, let
\begin{displaymath}
\sigma'_{\textup{max}}(r, d, b):=\max\big\lbrace \sigma\in\mathbb{R}_{>0}\big\vert\exists\lambda\in\Gamma(\textup{SL}_{r+1}(k))\textup{ s.t. }{E_{[\lambda], \sigma}^{d, b}}\cap \textup{Gr}(S_{d},b)\neq \emptyset \big\rbrace
\end{displaymath}
where the Grassmannian $\textup{Gr}(S_{d}, b)$ is considered as the closed subscheme of $\mathbb{P}(\wedge^{b} S_{d})$ via the Pl\"ucker embedding. A point $y\in\textup{Gr}(S_{d}, b)$ is a worst unstable point of $\textup{Gr}(S_{d}, b)$ if $y \in E_{[\lambda], \sigma'_{\textup{max}}(r, d, b)}^{d, b}$ for some 1-parameter subgroup $\lambda$ of $\textup{SL}_{r+1}(k)$.
\end{definition}

Describing worst unstable Hilbert points is a first step to understand the geometric meaning of the Hesselink stratification above. When $r=1$ and $P$ is a constant polynomial, a hypersurface defined by a homogeneous polynomial $f$ of degree $d$ is unstable if and only if there is a root of multiplicity $m\geq d/2$, which is explained in \cite[p. 80]{GIT}. It is natural to ask if a projective scheme represented by a worst unstable point has a unique closed point if $P$ is a constant polynomial. We will show that this guess is true for arbitrary $r\geq 1$ (Theorem~\ref{constant}). On the other hand, there is a theorem on a semi-stable bi-canonical curve.
\begin{theorem}[{\cite[Corollary 4.5. and 2.5. on page 924]{Hyeon}}]
\label{hyeon}
Suppose that $\textup{char}$ $k=0$. If $\mathcal{C}\subset \mathbb{P}^{3g-4}_{k}$ is a bi-canonical curve of genus $g\geq 3$ and is semi-stable for all but finitely many choices of Pl\"ucker embeddings, then $\mathcal{O}_{\mathcal{C}}$ is 2-regular.
\end{theorem}
The preceding theorem means that there is an upper bound for Castelnuovo-Mumford regularity $\textup{reg }\mathcal{O}_{\mathcal{C}}\leq 2$ for asymptotically semi-stable bi-canonical curves $\mathcal{C}\subset\mathbb{P}^{3g-4}_{k}$ of genus $g\geq 3$. Since $\textup{reg }\mathcal{O}_{C}\leq \textup{reg }\mathcal{I}_{C}$, it is reasonable to guess that every asymptotically worst unstable point attains maximal Castelnuovo-Mumford regularity, as lex-segment ideals do \cite[(2.9)]{Gotzmann}. Actually, we will show that this guess is true in the case of plane curves, which does not hold for arbitrary $r$ and $P$ (Theorem~\ref{regularity}).

We first try to describe the worst unstable points of $\textup{Gr}(S_{d}, Q_{P}(d))$ for $d\gg 0$ using the asymptotic behavior (Lemma~\ref{truncate1}) of some functions associated to the Hilbert polynomial P (Lemma~\ref{opt1}, Lemma~\ref{general1}). Also, we prove that these points have constant Hilbert polynomial. That is, a worst unstable point of $\textup{Gr}(S_{d}, Q_{P}(d))$ belongs to $\textrm{Hilb}^{P}(\mathbb{P}^{r}_{k})$ if and only if $P$ is constant (Theorem~\ref{useless}).
After that, we use \cite[Proposition 8]{Murai} to describe worst unstable Hilbert points when 
\begin{equation}
\label{goodsit}
P(d)=\binom{r+d}{r}-\binom{r+d-\gamma}{r}+p
\end{equation}
for some $p, \gamma\in\mathbb{N}$ (Theorem~\ref{constant}, Theorem~\ref{opt2}). Note that \eqref{goodsit} is always true when $r=2$. As we described above, we need to choose a Pl\"ucker embedding to define the Hesselink stratification of a Hilbert scheme $\textup{Hilb}^{P}(\mathbb{P}^{r}_{k})$. However, we will show that the set of all worst unstable Hilbert points with respect to $d$ remains unchanged for all but finitely many choices of $d$ when \eqref{goodsit} is true (Theorem~\ref{unchanged}). Furthermore, these worst unstable Hilbert points are Borel-fixed, so that it is possible to compute the Castelnuovo-Mumford regularity of an arbitrary worst unstable point of $\textup{Hilb}^{P}(\mathbb{P}_{k}^{r})$, which only depends on $r$ and $P$ (Theorem~\ref{regularity}). Under the condition \eqref{goodsit}, we will see that the Hilbert polynomial $P$ satisfies one of the following properties
\begin{itemize}
\item $r=2$ and $\gamma\neq 0$
\item $p\leq 2$
\end{itemize}
if and only if a worst unstable Hilbert point of $\textup{Hilb}^{P}(\mathbb{P}_{k}^{r})$ attains the maximal regularity $g_{P}$ in the last paragraph of this paper.

 By \cite[Section 2.3]{Donaldson}, asymptotic GIT stability and $K$-stability are closely related. The worst unstable Hilbert points described in this paper are worst points in the sense of asymptotic Hilbert-Mumford stability. The value of the equation on \cite[line 27, p. 11]{GTian} for a fixed adapted 1-parameter subgroup $\lambda$ of a worst unstable Hilbert point $x$ of $\textup{Hilb}^{P}(\mathbb{P}_{k}^{r})$ has to be {\it maximal} at $x$ for all $n\gg 0$. The dimension of the fiber of a maximal value under the function which maps each Hilbert point to corresponding $F_{0}(\lambda)$ \cite[p. 11]{GTian} can be large enough so that we can expect that the projective scheme represented by an arbitrary worst unstable Hilbert point may have the largest $F_{0}(\lambda)$ and $F_{1}(\lambda)$ (which is a Donaldson-Futaki invariant defined in \cite[p. 12]{GTian}) so that such a scheme may not be $K$-stable. We will compute some Donaldson-Futaki invariants of the projective schemes represented by the worst unstable points described in this paper and their associated 1-parameter subgroup using the asymptotic behavior of the numerical functions we have found. In this way, we will see that the projective scheme represented by a worst unstable Hilbert point is not $K$-stable if the Hilbert polynomial is in the form \eqref{goodsit} with the assumption $\gamma\neq 1$ in Theorem~\ref{Kinstability}.
\section{Preliminaries and Details in Computation}
\subsection{Castelnuovo-Mumford regularity and Gotzmann theorems}
A coherent sheaf $\mathcal{F}$ on $\mathbb{P}^{r}_{k}$ is $m$-regular if
\begin{displaymath}
\textup{H}^{i}(\mathbb{P}^{r}_{k}, \mathcal{F}(m-i))=0
\end{displaymath}
for all $i>0$. Let $\textup{reg}(\mathcal{F}):=\min\{m\in\mathbb{Z}|\mathcal{F}\textrm{ is }m\textrm{ regular}\}$. For an arbitrary graded $S$-module $M$, the regularity of $M$, denoted $\textup{reg}~M$ is defined to be the least integer $m$ satisfying
\begin{displaymath}
\textup{Ext}^{i}(M, S)_{j}=0\quad\textrm{ for all }i+j<-m.
\end{displaymath}
Equivalently, $\textup{reg}~M=\textup{max}_{j\geq 0} (b_{j}-j)$ where $b_{j}$ is the maximal degree of a minimal generator of $F_{i}$ for a minimal free resolution $\{F_{i}\}_{i=0}^{\infty}$ of $M$ \cite[20.5]{Eisenbudf}. There is an ideal sheaf $\tilde{I}$ on $\textup{Proj }S\cong \mathbb{P}_{k}^{r}$ associated to $I$ when $I$ is a graded ideal of $S$. There is an equality
\begin{displaymath}
\textup{reg}~I=\textup{reg}~\tilde{I}
\end{displaymath}
for every saturated graded ideal $I$ of $S$, as stated in \cite[Chapter 4]{Eisenbud}. The Gotzmann number $g_{P}$ (which is equal to $m(Q_{P})$ under the notation of \cite[p. 62]{Gotzmann}) is defined by the equation
\begin{displaymath}
\begin{split}
g_{P}:=&\max\left\{\textup{reg}(I)| I\textrm{ is a saturated graded ideal of }S,\right.\\
     &\qquad\qquad\quad\left.\textrm{ the Hilbert polynomial of } S/I\textrm{ is }P \right\}.
\end{split}
\end{displaymath}

Gotzmann's regularity theorem \cite[(2.9)]{Gotzmann} implies that lex-segment ideals whose Hilbert polynomial is $Q_{P}$ attain maximal regularity $g_{P}$. Actually, every Hilbert polynomial $Q_{P}$ of a graded ideal has a Macaulay representation \cite[(2.5)]{Gotzmann} of the form
\begin{displaymath}
Q_{P}(d)=\sum_{i=0}^{n}\binom{r-i+d-a_{i}}{r-i}
\end{displaymath}
for an integer $0\leq n<r$ and $n+2$ integers $\{a_{i}\}_{i=-1}^{n}$ satisfying
\begin{displaymath}
1=a_{-1}\leq a_{0}\leq \ldots \leq a_{n}.
\end{displaymath}
Let $M_{d}$ be the set of all monomials in $S_{d}$ and fix the lexicographic ordering $\leq_{\textup{lex}}$ on $M_{d}$ with respect to the term order $x_{0}>x_{1}>\ldots >x_{r}$. Let $I$ be the ideal generated by monomials greater than or equal to
\begin{displaymath}
\mu=x_{0}^{a_{0}-1}x_{1}^{a_{1}-a_{0}}\ldots x_{n}^{a_{n}-a_{n-1}+1}
\end{displaymath}
with respect to $\leq_{\textup{lex}}$. Then, $I_{d}$ is spanned by monomials greater than or equal to $\mu x_{r}^{d-a_{n}}$ with respect to $\leq_{\textup{lex}}$ for all $d\geq a_{n}$. That is,
\begin{displaymath}
I_{d}=\bigoplus_{i=0}^{n} \textup{span} \{m\in M_{d}| \textup{deg}_{x_{i}}m\geq a_{i}-a_{i-1}+1\textrm{ and }\textup{deg}_{x_{j}}m=a_{j}-a_{j-1}\textrm{ for all }j < i\}
\end{displaymath}
\begin{equation}
\label{gen}
=\bigoplus_{i=0}^{n} \left[ x_{i}^{a_{i}-a_{i-1}+1}\prod_{j=0}^{i-1}x_{j}^{a_{j}-a_{j-1}}\right] k[x_{i}, \ldots, x_{r}]_{d-a_{i}}
\end{equation}

for all $d\geq g_{P}$. Thus we can directly compute
\begin{displaymath}
Q_{P}(d)=\dim_{k}I_{d}=\sum_{i=0}^{n}\binom{r-i+d-a_{i}}{r-i}
\end{displaymath}
for all $d\geq g_{P}$. Equation \eqref{gen} implies that $I$ has a minimal generator containing $\mu$ of degree $a_{n}$. Also, $I$ is saturated. Therefore, $g_{P}\geq \textup{reg }I\geq a_{n}$. Furthermore, $g_{P}\leq a_{n}$ by \cite[(2.9)]{Gotzmann} so that $a_{n}=g_{P}$.

\begin{definition}
Let $b\in\mathbb{N}$. For $W\in\textup{Gr}(S_{d}, b)$, let $SW$ denote the ideal of $S$ generated by $W$. If $C\subset S_{d}$, we let $\textup{span}\textrm{ }C$ denote the $k$-subspace of $S_{d}$ spanned by $C$.  Define $U_{i}(n)=\{v\in M_{d}| \textrm{deg}_{x_{j}}v \leq \textrm{deg}_{x_{j}}n \textrm{ for all } j\in\{0, 1,,
\ldots, r\}\setminus\{ i\}\}$ for a monomial $n\in S_{d}$ and $i\in\{0, 1, \ldots, r\}$. A monomial ideal $J$ is said to be Borel-fixed if $\frac{x_{i}}{x_{i+1}}m$ is not a monomial or is in $J$, for all monomials $m\in J$ and $0\leq i<r$.
\end{definition}
\subsection{State polytopes and geometric invariant theory}
Let $T_{r}$ denote the maximal torus of $\textup{GL}_{r}(k)$, which consists of all diagonal matrices in $\textup{GL}_{r}(k)$. For any affine algebraic group $G$, let $X(G)$  (resp. $\Gamma(G)$) be the group of characters (resp. 1-parameter subgroups) of $G$. Consider the canonical $\textup{GL}_{r+1}(k)$-actions on a Pl\"ucker coordinate $\mathbb{P}\left(\wedge^{b} S_{d}\right)$ and its affine cone $\wedge^{b}S_{d}$, which are induced by the canonical $\textup{GL}_{r+1}(k)$-action on $S_{1}$. These actions induce a $T_{r+1}$-action on $\wedge^{b}S_{d}$, which has the character decomposition \cite[Proposition 4.14]{Mukai}
\begin{displaymath}
\wedge^{b}S_{d}=\bigoplus_{\chi\in X(T_{r+1})} \left(\wedge^{b}S_{d}\right)_{\chi}
\end{displaymath}
where
\begin{displaymath}
\left(\wedge^{b}S_{d}\right)_{\chi}=\left\{v\in \wedge^{b}S_{d}| t.v=\chi(t)v \textrm{ for all }t\in T_{r+1} \right\}.
\end{displaymath}
 Let us also fix a basis $\{\chi_{i}\}_{i=0}^{r}$ of $X(T_{r+1})$ where $\chi_{i}(t)=t_{ii}$ for all $0\leq i\leq r$ and $t\in T_{r+1}$. Then, we can easily see that $\left(\wedge^{b}S_{d}\right)_{\chi}$ is generated by
\begin{displaymath}
\left\{ \wedge_{i=1}^{b} m_{i}\Bigg\vert m_{i}\in M_{d}\textrm{ for all }i\textrm{ and }\prod_{i=1}^{b}m_{i}=\prod_{j=0}^{r} x_{j}^{d_{j}} \right\}
\end{displaymath}
if $\chi=\prod_{j=0}^{r}\chi_{j}^{d_{j}}$. It follows that $\sum_{j=0}^{r}d_{i}=db$ if $\left(\wedge^{b}S_{d}\right)_{\chi}\neq 0$. There is a basis $\{\lambda_{i}\}_{i=0}^{r}$ of $\Gamma(T_{r+1})$ which is the dual basis of $\{\chi_{i}\}_{i=0}^{r}$ with respect to the pairing $\langle ,\rangle:X(T_{r+1})\times\Gamma(T_{r+1})\rightarrow \mathbb{Z}$ which satisfies
\begin{displaymath}
\chi(\lambda(t))=t^{\langle \chi, \lambda \rangle}
\end{displaymath}
for all $t\in k\setminus \{0\}$. We can identify $X(T_{r+1})\cong \mathbb{Z}^{r+1}\cong \Gamma(T_{r+1})$ under this choice of basis. Let $\Vert\cdot\Vert$ denote the Euclidean norm on $X(T_{r+1})\cong \mathbb{Z}^{r+1}\cong \Gamma(T_{r+1})$ with respect to the basis $\{\chi_{i}\}_{i=0}^{r}$ of $X(T_{r+1})$. There is a norm $\Vert\cdot\Vert_{\mathbb{R}}$ on $X(T_{r+1})_{\mathbb{R}}=X(T_{r+1})\otimes_{\mathbb{Z}}\mathbb{R}\cong \mathbb{R}^{r+1}\cong \Gamma(T_{r+1})\otimes_{\mathbb{Z}}\mathbb{R}=\Gamma(T_{r+1})_{\mathbb{R}}$ induced by $\Vert\cdot\Vert$. There is also a pairing $\langle, \rangle_{\mathbb{R}}:X(T_{r+1})_{\mathbb{R}}\times\Gamma(T_{r+1})_{\mathbb{R}}\rightarrow \mathbb{R}$ obtained from $\langle\textrm{ },\textrm{ }\rangle$ by the base change to $\mathbb{R}$. 
\begin{definition}
An arbitrary $v\in\wedge^{b}S_{d}$ has a decomposition
\begin{displaymath}
v=\sum_{\chi\in X(T_{r+1})} v_{\chi}
\end{displaymath}
with $v_{\chi}\in \left(\wedge^{b}S_{d}\right)_{\chi}$. The state $\Xi_{[v]}$of the line $[v]\in \mathbb{P}(\wedge^{b}S_{d})$ through the origin and $v$ is the set
\begin{displaymath}
\Xi_{[v]}=\left\{\chi\in X(T_{r+1})| v_{\chi}\neq 0 \right\}.
\end{displaymath}
The state polytope $\Delta_{[v]}$ of $[v]$ is the convex hull of $\Xi_{[v]}\otimes_{\mathbb{Z}}1$ in $X(T_{r+1})_{\mathbb{R}}$.
\end{definition}

 For example, $\Delta_{[v]}$ is a point if $v$ is a wedge of monomials. Let $\xi_{d, b}=\frac{db}{r+1}\mathbbm{1}$ where $\mathbbm{1}$ is the all-$1$ vector of $X(T_{r+1})_{\mathbb{R}}$ with respect to the  basis $\{\chi_{i}\}_{i=0}^{r}$. If $L$ is the  line bundle on $\textup{Hilb}^{P}(\mathbb{P}_{k}^{r})$ defined by the Pl\"ucker embedding corresponding to $d$, $L$ admits a canonical linearization by the canonical $\textup{SL}_{r+1}(k)$-action and $L^{\otimes (r+1)}$ also can be linearized by a similar way \cite[p. 33]{GIT}. We can see that the induced $\textup{GL}_{r+1}(k)$-linearization on $L^{\otimes(r+1)}$ twisted by the $dQ_{P}(d)$'th power of the determinant function on $\textup{GL}_{r+1}(k)$ (whose restriction on $T_{r+1}$ is the character corresponding to $(r+1)\xi_{d, Q_{P}(d)}$) and $\textup{SL}_{r+1}(k)$-linearization on $L$ are equivalent in the sense of GIT; they defines the same stable locus, semi-stable locus, GIT-quotients and same numerical weight functions on $\Gamma(\textup{SL}_{r+1}(k))$ up to constant \cite[Definition 2.2, p. 49]{GIT}. Note that $\xi_{d, b}$ is the arithmetic mean of the set $\{\chi\in X(T_{r+1})| \left( \wedge^{b} S_{d} \right)_{\chi}\neq 0\}\otimes_{\mathbb{Z}}1$. For each $v\in \wedge^{b}S_{d}$ satisfying $\xi_{d, b}\notin\Delta_{[v]}$, there is a unique $\lambda_{[v]}\in\Gamma(T_{r+1})$ which satisfies the following properties.
\begin{itemize}
\item There is $u\in \mathbb{R}^{+}$ such that $\lambda_{[v]}\otimes_{\mathbb{Z}}u+\xi_{d, b}\in\Delta_{[v]}$ via the isomorphism $X(T_{r+1})\cong\Gamma(T_{r+1})$ we defined above and
$\Vert\lambda_{[v]}\otimes_{\mathbb{Z}}u\Vert_{\mathbb{R}}$ is equal to the distance from $\xi_{d, b}$ to $\Delta_{[v]}$.
\item There is no $m\in\mathbb{N}\setminus\{0, 1\}$ and $\lambda\in\Gamma(T_{r+1})$ such that $m\lambda=\lambda_{[v]}$.
\end{itemize}
The image of such a 1-parameter subgroup $\lambda_{[v]}$ is contained in $\textup{SL}_{r+1}(k)$ because the sum of all coefficients of $\lambda_{[v]}$ is 0. 
\begin{definition}
Let $|\Delta_{[v]}|_{0}$ be the distance from $\xi_{d, b}$ to $\Delta_{[v]}$. The set of state polytopes $\{\Delta_{g.[v]}|g\in \textup{GL}_{r+1}(k)\}$ determines the {\it state} of $[v]$ from the viewpoint of geometric invariant theory.
\end{definition}
\begin{theorem}
\label{basic}
For an arbitrary $v\in\wedge^{b}S_{d}$, there is $g\in\textup{GL}_{r+1}$ such that
\begin{displaymath}
|\Delta_{g.[v]}|_{0}=\max_{h\in\textup{GL}_{r+1}(k)} |\Delta_{h.[v]}|_{0}.
\end{displaymath}
For such $g$, $[v]\in E_{[\lambda_{g.[v]}], |\Delta_{g.[v]}|_{0}}^{d, b}$ if $|\Delta_{g.[v]}|_{0}>0$. Otherwise, $[v]$ is semi-stable.
\end{theorem}
\begin{proof}
A generalized version of this theorem can be found in \cite[1.1.4. and 1.1.5.]{VGIT}. See also \cite{Kempf}, \cite{Hesselink} and \cite{Ian}.
\end{proof}

\subsection{Computation of worst unstable points}
For an arbitrary $v\in\wedge^{b}S_{d}$, let $|\Delta_{[v]}|$ denote the distance from the origin of $X(T_{r+1})_{\mathbb{R}}$ to $\Delta_{[v]}$. Since $\Delta_{[v]}\subset H_{d, b}:=\{w\in X(T_{r+1})_{\mathbb{R}}|\langle w, \mathbbm{1}\otimes_{\mathbb{Z}}1\rangle=db\}$ and $\xi_{d, b}$ is the point on $H_{d, b}$ closest to the origin, $|\Delta_{[v]}|^{2}=|\Delta_{[v]}|_{0}^{2}+\Vert\xi_{d, b}\Vert^{2}_{\mathbb{R}}=|\Delta_{[v]}|_{0}^{2}+\frac{d^{2}b^{2}}{r+1}$. Therefore, it is enough to consider the optimization problem on $|\Delta_{[v]}|$ to describe worst unstable points.
\begin{definition}
Define 
\begin{displaymath}
R(d, b, S)=\{W\in \textup{Gr}(S_{d}, b)|W\textrm{ is generated by monomials}\},
\end{displaymath}
\begin{displaymath}
Z_{d}^{b}(S)=\{W\in R(d, b, S)|\textrm{ }|\Delta_{W}|\geq |\Delta_{W'}|\textup{ for all }W'\in R(d, b, S)\}
\end{displaymath}
and
\begin{displaymath}
\begin{split}
X_{d}^{P}(S)=&\left\{W\in R(d, Q_{P}(d), S)\cap \textup{Hilb}^{P}(\mathbb{P}^{r}_{k})\big\vert\right.\\
&\quad\left.\textrm{ }|\Delta_{W}|\geq |\Delta_{W'}|\textup{ for all }W'\in R(d, Q_{P}(d), S)\cap \textup{Hilb}^{P}(\mathbb{P}^{r}_{k}) \right\}.
\end{split}
\end{displaymath}
We can see that both $Z_{d}^{b}(S)$ (resp. $X_{d}^{P}(S)$) and $\textup{GL}_{r+1}(k).Z_{d}^{b}(S)$ (resp. $\textup{GL}_{r+1}(k).$ $X_{d}^{P}(S)$) are closed subschemes of $\textup{Gr}(S_{d}, b)$ (resp. $\textup{Hilb}^{P}(\mathbb{P}_{k}^{r})$) under some scheme structure using the argument explained in \cite[(6.1)(c), (6.2)(b)]{Hesselink}.  Let $\log {W}\in X(T_{r+1})_{\mathbb{R}}$ denote the lattice points satisfying $\Delta_{W}=\{\log{W}\}$ for an arbitrary $W\in R(d, b, S)$.
\end{definition}

For all $d, b\in \mathbb{N}$ and $W\in R(d, b,S)$, let $N(W)$ denote the monomial basis of $W$. Let $W^{\star}\in\textup{Gr}(S_{d}, b')$ be the $k$-subspace of $S_{d}$ generated by $M_{d}\setminus N(W)$ where $b'=\binom{r+d}{r}-b$. If $W\in R(d, b, S)$ then $\log {W}$ records the exponent of a monomial
\begin{displaymath}
\prod_{n\in N(W)}n.
\end{displaymath}

Now we are ready to state a theorem on a construction of the set of worst unstable points from from our $Z_{d}^{b}(S)$ and $X_{d}^{P}(S)$.

\begin{theorem}
\label{foundation}
Fix a Hilbert scheme $\textup{Hilb}^{P}(\mathbb{P}^{r}_{k})$ and the Pl\"ucker embedding  corresponding to an integer $d\geq g_{P}$.
Every worst unstable point of $\textup{Gr}(S_{d}, b)$ (resp. $\textup{Hilb}^{P}(\mathbb{P}^{r}_{k})$ with respect to $d$) is in the orbit of some element in $Z_{d}^{b}(S)$ (resp. $X_{d}^{P}(S)$). In particular, the set of all worst unstable points of $\textup{Gr}(S_{d}, b)$ (resp. $\textup{Hilb}^{P}(\mathbb{P}_{k}^{r})$ with respect to $d$) is a closed subscheme of $\textup{Gr}(S_{d}, b)$ (resp. $\textup{Hilb}^{P}(\mathbb{P}_{k}^{r})$) under some scheme structure.
\end{theorem}
\begin{proof}
It is clear by the construction of $Z_{d}^{b}(S)$ and Theorem~\ref{basic} that an arbitrary point in $\textup{GL}_{r+1}(k).Z^{b}_{d}(S)$ is a worst unstable point of $\textup{Gr}(S_{d}, b)$. Conversely, any worst unstable point of $\textup{Gr}(S_{d}, b)$ is in the orbit of $W\in\textup{Gr}(S_{d}, b)$ such that the cardinality of $\Delta_{W}$ is $1$. Otherwise, we can find $v(g)\in R(d, b, S)$ satisfying $|\Delta_{v(g)}|>|\Delta_{g.W}|$ for an arbitrary choice of $g\in\textup{GL}_{r+1}(k)$ and such a conclusion contradicts the maximality of $\max_{g\in\textup{GL}_{r+1}(k)}|\Delta_{g.W}|$. Fixing any total order of the monomial basis of $S_{d}$, $W$ is the wedge of all row vectors in a matrix in echelon form. It directly follows that $W\in R(d, b, S)$ if the cardinality of $\Delta_{W}$ is $1$. Therefore, $\textup{GL}_{r+1}(k).Z^{b}_{d}(S)$ is equal to the set of worst unstable points of $\textup{Gr}(S_{d}, b)$ as sets. Let $I$ be a worst unstable Hilbert point of $\textup{Hilb}^{P}(\mathbb{P}^{r}_{k})$ with respect to $d$. We may consider $I$ as a saturated graded ideal of $S$, whose Hilbert polynomial is $Q_{P}$ as an $S$-module. Without loss of generality, assume that $|\Delta_{I}|=\max_{g\in\textup{GL}_{r+1}(k)} |\Delta_{g.I}|$. If the cardinality of $\Delta_{I}$ is not equal to $1$, then we can choose a vertex $x$ of $\Delta_{I}$, satisfying $\Vert x \Vert_{\mathbb{R}}>|\Delta_{I}|$. Using the proof of \cite[Theorem 3.1.]{Bayer}, we can show that there is a monomial order $\prec$ on $S_{d}$, such that $\Delta_{\textup{in}_{\prec}I}=\{x\}$. The Hilbert polynomial of $\textup{in}_{\prec}I$ is equal to $Q_{P}$ and $|\Delta_{\textup{in}_{\prec}I}|:=\Vert x \Vert_{\mathbb{R}}>|\Delta_{I}|$ so that $I$ is not a worst unstable point of $\textup{Hilb}^{P}(\mathbb{P}_{k}^{r})$ with respect to $d$ by Theorem~\ref{basic}, which contradicts the assumption on $I$, so that the cardinality of $\Delta_{I}$ is equal to $1$. The remaining claims easily follow from the construction of $X_{d}^{P}(S)$ and Theorem~\ref{basic}.  
\end{proof}
From now on, we will concentrate on the computation of $Z_{d}^{b}(S)$ and $X_{d}^{P}(S)$. Suppose $W\in R(d, Q_{P}(d), S)$ and $\log {W}=(c_{0}, c_{1}, \ldots, c_{r})$. Then $\sum_{i=0}^{r}c_{i}=dQ_{P}(d)$ from the definition. Also, we can derive $|\Delta_{W}|^{2}=\sum_{i=0}^{r}c_{i}^{2}$ and
\begin{displaymath}
\prod_{n\in N_{W}} n=\prod_{i=0}^{r} x_{i}^{c_{i}}.
\end{displaymath}
The function
\begin{displaymath}
f(c_{0}, c_{1}, \ldots, c_{r})=\sum_{i=0}^{r}c_{i}^{2}
\end{displaymath}
defined on the set $\{\{c_{i}\}_{i=0}^{r}\in (\mathbb{R}^{+})^{r+1}|\sum_{i=0}^{r}c_{i}=dQ_{P}(d)\}$ has a unique minimum at $\frac{dQ_{P}(d)}{r+1}\mathbbm{1}$ by the convexity of $f$ and has a maximum at $(dQ_{P}(d), 0, \ldots, 0)$. Therefore, it is natural to guess that $W$ maximalizing $\max_{0\leq i\leq r} c_{i}$ also maximalizes $|\Delta_{W}|$. It is straightforward to check that all lex-segment ideals maximalize $\max_{0\leq i\leq r} c_{i}$. However, it is not true that the orbit of a lex-segment ideal is the set of all worst unstable Hilbert points. 
\begin{example}
Let $S=k[x, y, z]$, $P(d)=3$ and $d=3$ so that $r=2$ and $Q_{P}(d)=7$. It is true that the orbit of $I=\langle x^{3}, x^{2}y, x^{2}z \rangle$ is the set of all worst unstable points of the Grassmannian $\textup{Gr}(S_{3}, 3)\subset\mathbb{P}(S_{3}\wedge S_{3}\wedge S_{3})$ because such a choice of $I$ maximalizes the first coordinate of $\log{I_{3}}$. Indeed,
\begin{displaymath}
(a+1)^{2}+(b+1)^{2}+(c+1)^{2}\leq (a+b+c+1)^{2}+1+1=51
\end{displaymath}
for all $a, b, c\geq 0$ satisfying $a+b+c=6$ and equality holds if and only if two among $a$, $b$ and $c$ are zero. Therefore, $I_{3}$ is an element of the set
\begin{displaymath}
R(3, 3, S)\setminus \left(R(3, 3, k[x, y])\cup R(3, 3, k[y, z])\cup R(3, 3, k[z, x])\right),
\end{displaymath}
which maximalizes the function $|\Delta_{\cdot}|$. Also, we can check that
\begin{displaymath}
(a+1)^{2}+(b+3)^{2}\leq (a+b+1)^{2}+9 <51
\end{displaymath}
for all $a, b\geq 0$ satisfying $a+b=5$; this means that the orbit of $I$ is the unique worst unstable orbit of $\textup{Gr}(S_{3}, 3)$ by Theorem~\ref{foundation}.
 By the symmetry stated in Lemma~\ref{duality}, if 
\begin{displaymath}
J=\langle xy^{2}, xyz, xz^{2}, y^{3}, y^{2}z, yz^{2}, z^{3} \rangle
\end{displaymath}
then the orbit of $J_{3}$ is the set of all worst unstable points of $\textup{Gr}(S_{3}, 7)$. Moreover, $J$ has the minimal growth at degree $3$. Thus, $J_{3}\in \textrm{Hilb}^{3}(\mathbb{P}^{2}_{k})$ and it is a worst unstable Hilbert point.
\end{example}

In fact, we will generalize this observation; our computations so far can be generalized if we choose the Pl\"ucker embedding corresponding to an integer $d\gg0$.

\section{Worst unstable points of a Grassmannian containing a Hilbert scheme}

As we discussed in the previous section, there is a relation between worst unstable points of $\textup{Gr}(S_{d}, P(d))$ and worst unstable points of $\textup{Gr}(S_{d}, Q_{P}(d))$.
\begin{lemma}
\label{duality}
For $W\in R(d, P(d), S)$ we have that $W\in Z_{d}^{P(d)}(S)$ if and only if $W^{\star}\in Z_{d}^{Q_{P}(d)}(S)$ . 
\end{lemma}
\begin{proof}
For any $V\in R(d, P(d), S)$, we have 
\begin{equation}
\label{dualform}
\log{V}+\log{V^\star}=\frac{d}{r+1}\binom{r+d}{r}(\mathbbm{1}
\otimes_{\mathbb{Z}}1).
\end{equation}
Also, we have the equalities
\begin{displaymath}
\begin{split}
\max_{V\in R(d, P(d), S)}&\bigg\Vert\log{V}-\frac{dP(d)}{r+1}(\mathbbm{1}
\otimes_{\mathbb{Z}}1)\bigg\Vert_{\mathbb{R}}^{2}\\
&=\max_{V\in R(d, P(d), S)}\bigg\Vert\log{V^\star}-\frac{dQ_{P}(d)}{r+1}(\mathbbm{1}
\otimes_{\mathbb{Z}}1)\bigg\Vert_{\mathbb{R}}^{2}\\
&=\max_{V\in R(d, Q_{P}, d)(S)}\bigg\Vert\log{V}-\frac{dQ_{P}(d)}{r+1}(\mathbbm{1}
\otimes_{\mathbb{Z}}1)\bigg\Vert_{\mathbb{R}}^{2}
\end{split}
\end{displaymath}
and
\begin{displaymath}
\bigg\Vert\log{W^{\star}}-\frac{dQ_{P}(d)}{r+1}(\mathbbm{1}
\otimes_{\mathbb{Z}}1)\bigg\Vert_{\mathbb{R}}=\bigg\Vert\log{W}-\frac{dP(d)}{r+1}(\mathbbm{1}
\otimes_{\mathbb{Z}}1)\bigg\Vert_{\mathbb{R}}
\end{displaymath}
using \eqref{dualform}. Then $W\in Z_{d}^{P(d)}(S)$ if and only if 
\begin{equation}
\label{dual}
 \bigg\Vert \log{W}-\frac{dP(d)}{r+1}(\mathbbm{1}
\otimes_{\mathbb{Z}}1)\bigg\Vert^{2}_{\mathbb{R}} = \max_{V\in R(d, P(d), S)}\bigg\Vert\log{V}-\frac{dP(d)}{r+1}(\mathbbm{1}
\otimes_{\mathbb{Z}}1)\bigg\Vert_{\mathbb{R}}^{2}
\end{equation} 
because $\Delta_{V}\subset \{\mathbf{v}\in  X(T)_{\mathbb{R}}|\langle\mathbf{v},(\mathbbm{1}
\otimes_{\mathbb{Z}}1)\rangle_{\mathbb{R}} =dP(d)\}$ for all $V\in \textup{Gr}(S_{d}, P(d))$. Similarly, $W^{\star}\in Z_{d}^{Q_{P}(d)}(S)$ if and only if
\begin{displaymath}
\bigg\Vert\log{W^{\star}}-\frac{dQ_{P}(d)}{r+1}(\mathbbm{1}
\otimes_{\mathbb{Z}}1)\bigg\Vert_{\mathbb{R}}^{2} = \max_{V\in R(d, Q_{P}(d), S)}\bigg\Vert\log{V}-\frac{dQ_{P}(d)}{r+1}(\mathbbm{1}
\otimes_{\mathbb{Z}}1)\bigg\Vert_{\mathbb{R}}^{2}.
\end{displaymath}
This completes the proof.
\end{proof}
We will examine the asymptotic behavior of some functions associated to a Hilbert polynomial $P$ in Lemma~\ref{truncate1}. This observation (i.e., Lemma 3.3) makes it possible to describe worst unstable points of $\textup{Gr}(S_{d}, Q_{P}(d))$ for $d\gg 0$. If $P=0$ or $Q_{P}=0$ then our problem becomes a trivial one. Let's assume that $P\neq 0\neq Q_{P}$.
\begin{definition}
Let $\mu(t, d)$ be the $t$'th greatest monomial of $M_{d}$ with respect to $\leq_{\textup{lex}}$. Let $L(t, d, S)$ be the subspace of $S_{d}$ generated by generated by $\{\mu(i, d)|1\leq i\leq t\}$ and $A(t, d, S)$ be the subspace of $S_{d}$ generated by $\{\mu(i, d)|t+1\leq i\leq \binom{r+d}{d}\}$ We may also consider $L(t, d, S)$ and $A(t, d, S)$ as points in the Grassmannian using Pl\"ucker embedding:
\begin{displaymath}
L(t, d, S)=\left[\bigwedge_{i=1}^{t}\mu(i, d)\right]\in \textup{Gr}(S_{d}, t)
\end{displaymath}
and
\begin{displaymath}
A(t, d, S)=\left[\bigwedge_{i=t+1}^{\binom{r+d}{r}}\mu(i, d)\right]\in \textup{Gr}\left( S_{d}, \binom{r+d}{r}-t\right).
\end{displaymath}
\end{definition}

Consider $\textrm{Hilb}^{P}(\mathbb{P}^{r}_{k})$ for a Hilbert polynomial $P\in\mathbb{Q}[t]$ where $t$ is a variable. We can define some numerical functions and constants corresponding to $P$.

\begin{definition}
There is a function $\delta:\mathbb{N}\rightarrow\mathbb{N}$ such that 
\begin{displaymath}
\binom{r+\delta(d) -1}{r}<Q_{P}(d)\leq \binom{r+\delta(d)}{r}
\end{displaymath}
for all $d\geq g_{P}$ because $\mathbb{Z}$ is well-ordered. The defining inequality of $\delta$ is equivalent to 
\begin{displaymath}
x_{0}^{d-\delta(d)}x_{r}^{\delta(d)}\leq_{\textup{lex}} \mu(Q_{P}(d), d)<_{\textup{lex}} x_{0}^{d-\delta(d)+1}x_{r}^{\delta(d)-1}.
\end{displaymath}
That is, $d-\delta(d)=\deg_{x_{0}}\mu(Q_{P}(d), d)\geq 0$ for all $d\geq g_{P}$.
Also, $x_{r}\mu(Q_{P}(d), d)=\mu(Q_{P}(d+1), d+1)$ since the Hilbert polynomial of the ideal generated by $L(Q_{P}(d), d, S)$ is $Q_{P}$ for all $d\geq g_{P}$, as we can see in \cite[(2.1), (2.5), (2.9)]{Gotzmann}. Therefore, 
\begin{displaymath}
\begin{split}
d+1-\delta(d+1)&=\deg_{x_{0}}\mu(Q_{P}(d+1), d+1)\\
&=\deg_{x_{0}}x_{r}\mu(Q_{P}(d), d)\\
&=\deg_{x_{0}}\mu(Q_{P}(d), d)\\
&=d-\delta(d)
\end{split}
\end{displaymath}
for all $d\geq g_{P}$ so that there's an integer $\gamma\geq 0$ such that $\delta(d)=d-\gamma$ for all $d\geq g_{P}$. There is a function $l:\mathbb{N}\rightarrow\mathbb{Z}$ such that
\begin{displaymath}
\begin{split}
\dim_{k}L\left(\binom{r+l(d)-1}{r}, d, S\right)&=\sum_{0\leq i\leq l(d)-1}{\binom{r+i-1}{r-1}}\\
&< P(d)\\
&\leq \sum_{0\leq i\leq l(d)}{\binom{r+i-1}{r-1}}\\
&=\dim_{k}L\left(\binom{r+l(d)}{r}, d, S\right)
\end{split}
\end{displaymath}
for $d\geq g_{P}$ because $\mathbb{Z}$ is well-ordered. Define
\begin{displaymath}
\begin{split}
e(d)&=\bigg\langle \log{L(P(d), d, S)}, \sum_{i=1}^{r} \lambda_{i}\otimes_{\mathbb{Z}}1 \bigg\rangle_{\mathbb{R}}\\
&=|\{ n\in N(L(P(d), d, S))| \deg_{x_{0}} n=d-l(d) \}|l(d)\\
&\quad +\sum_{0\leq i\leq l(d)-1}|\{ n\in N(L(P(d), d, S))| \deg_{x_{0}} n =d-i \}|i\\
&=\bigg[P(d)-\sum_{0\leq i\leq l(d)-1}{\binom{r+i-1}{r-1}}\bigg]l(d)+\sum_{0\leq i\leq l(d)-1}{\binom{r+i-1}{r-1}}i.
\end{split}
\end{displaymath}
That is, 
\begin{equation}
\label{construction}
\bigg\langle \log{L(P(d), d, S)},  \lambda_{0}\otimes_{\mathbb{Z}}1 \bigg\rangle_{\mathbb{R}}=dP(d)-e(d).
\end{equation}
Indeed, $L(P(d),d, S)$ maximalizes the function
\begin{displaymath}
\max_{0\leq i\leq r}\langle \log{\overline{\quad}},  \lambda_{i}\otimes_{\mathbb{Z}}1 \rangle_{\mathbb{R}}:R(d, P(d), S)\rightarrow \mathbb{N}.
\end{displaymath}
\end{definition}

\begin{lemma}
\label{maximality}
$dP(d)-e(d)\geq \langle \log{W},  \lambda_{0}\otimes_{\mathbb{Z}}1 \rangle_{\mathbb{R}}$ for all $W\in R(d, P(d), S)$. In the case of equality, we have $L(\binom{r+l(d)-1}{r}, d, S)\subset W$ and $W\subset L(\binom{r+l(d)}{r}, d, S)$.
\end{lemma}
\begin{proof}
Fix $W\in R(d, P(d), S)$. If $L(\binom{r+l(d)-1}{r}, d, S)\not\subset W$ or $W\not\subset L(\binom{r+l(d)}{r}, d, S)$, then we can find monomials $n\in M_{d}\setminus N(W)$ and $m\in N(W)$ such that
\begin{displaymath}
\langle \log{W'},  \lambda_{0}\otimes_{\mathbb{Z}}1 \rangle_{\mathbb{R}}>\langle \log{W},  \lambda_{0}\otimes_{\mathbb{Z}}1 \rangle_{\mathbb{R}}
\end{displaymath}
by the construction of $l$, where $W'$ is generated by $N(W)\cup\{n\}\setminus \{m\}$.\newline If $L(\binom{r+l(d)-1}{r}, d, S)\subset W$ and $W\subset L(\binom{r+l(d)}{r}, d, S)$, then we have $dP(d)-e(d) = \langle \log{W},  \lambda_{0}\otimes_{\mathbb{Z}}1 \rangle_{\mathbb{R}}$ by the construction of $e$.
\end{proof}
Before we state another lemma about asymptotic behavior of numerical functions defined above in this section, let's define the first {\it discriminant} function $\Phi$.
\begin{definition}Let $\Phi:\mathbb{N}\rightarrow\mathbb{N}$ be the function satisfying
\begin{displaymath}
\Phi(d)=d^{2}[P(d)]^{2}-4dP(d)e(d)+\frac{2(r+1)}{r}[e(d)]^{2}
\end{displaymath}
for all $d\in\mathbb{N}$.
\end{definition}
Actually, $\Phi$ is the discriminant of a quadratic inequality, which will be mentioned in Lemma~\ref{opt1}.
\begin{lemma}
\label{truncate1}
There is an integer $D_{P}\geq g_{P}$ corresponding to the Hilbert polynomial $P$ such that every integer $d\geq D_{P}$ satisfies following properties.
\begin{equation}
\label{discriminant1}
\Phi(d)>0
\end{equation}
\begin{equation}
\label{lowerbound1}
\frac{dP(d)-\sqrt{\Phi(d)}}{2dP(d)}<\frac{1}{r+1}
\end{equation}
\begin{equation}
\label{upperbound1}
\bigg|dP(d)-2e(d)-\sqrt{\Phi(d)}\bigg|<2e(d)
\end{equation}
\begin{equation}
\label{asymptotic1}
\frac{e(d)}{dP(d)}\leq \frac{l(d)}{d}<\frac{1}{8}
\end{equation}
Furthermore, if $P$ is a constant polynomial, then there's $D_{P}\in\mathbb{N}$ such that $d\geq D_{P}$ implies \eqref{discriminant1}, \eqref{lowerbound1}, \eqref{asymptotic1} and
\begin{equation}
\label{upperbound2}
\bigg|dP(d)-2e(d)-\sqrt{\Phi(d)}\bigg|<2.
\end{equation}
\end{lemma}
\begin{proof}
We claim that
\begin{displaymath}
\lim_{d\rightarrow\infty}\frac{l(d)}{d}=0
\end{displaymath}
and this implies
\begin{displaymath}
\lim_{d\rightarrow\infty}\frac{e(d)}{dP(d)}=0.
\end{displaymath}
These properties imply what we want to prove. For example, the left-hand sides of \eqref{lowerbound1} and \eqref{upperbound2} tend to zero as $d\rightarrow\infty$ and the left-hand side of \eqref{discriminant1} tends to $\infty$ as $d\rightarrow\infty$ if these assumptions are true. By the definition of $l$, the value $l(d)$ is the smallest integer satisfying
\begin{displaymath}
P(d) \leq \sum_{i=0}^{l(d)}{\binom{r+i-1}{r-1}}=\binom{r+l(d)}{r}.
\end{displaymath} 
Let $\frac{r-1}{r}<\eta<1$. Note that
\begin{displaymath}
\lim_{d\rightarrow\infty}{\sum_{i=\delta(d)}^{d}\frac{\binom{r+i-1}{r-1}}{\binom{r+\lceil d^{\eta}\rceil}{r}}}= \lim_{d\rightarrow\infty}{\sum_{i=0}^{\gamma}\frac{\binom{r+d-\gamma +i-1}{r-1}}{\binom{r+\lceil d^{\eta}\rceil}{r}}} =0.
\end{displaymath}
The preceding formula means that
\begin{displaymath}
\begin{split}
P(d)&<\binom{r+d}{r}-\binom{r+\delta(d)-1}{r}\\
&=\sum_{i=\delta(d)}^{d}{\binom{r+i-1}{r-1}}\\
&< \binom{r+\lceil d^{\eta}\rceil}{r}
\end{split}
\end{displaymath}
for sufficiently large $d$. Therefore, $l(d)\leq \lceil d^{\eta}\rceil$ for $d\gg 0$. From the definition of $e$ we see that 
\begin{displaymath}
\begin{split}
e(d)&=\sum_{0\leq i\leq l(d)-1}{\binom{r+i-1}{r-1}}i+l(d)\bigg[P(d)-\sum_{0\leq i\leq l(d)-1}{\binom{r+i-1}{r-1}}\bigg]\\
&\leq \sum_{0\leq i\leq l(d)-1}{\binom{r+i-1}{r-1}}l(d)+l(d)\bigg[P(d)-\sum_{0\leq i\leq l(d)-1}{\binom{r+i-1}{r-1}}\bigg]\\
&=l(d)P(d)
\end{split}
\end{displaymath}
so that
\begin{displaymath}
0\leq\lim_{d\rightarrow\infty}{\frac{e(d)}{dP(d)}}\leq\lim_{d\rightarrow\infty}{\frac{l(d)}{d}}\leq\lim_{d\rightarrow\infty}{\frac{\lceil d^{\eta}\rceil}{d}}=0,
\end{displaymath}
as desired.
\end{proof}

We can define the number $D_{P}$ corresponding to the Hilbert polynomial $P$ using Lemma 3.3.
\begin{definition}
For each Hilbert polynomial $P$, let $D_{P}$ be the minimal integer satisfying the conditions in Lemma~\ref{truncate1}. 
\end{definition}

Now we are ready to state some properties of $Z_{d}^{P(d)}(S)$ for an arbitrary constant  Hilbert polynomial $P$. 

\begin{lemma}
\label{opt1}
Let $P$ be a constant Hilbert polynomial and suppose that $d\geq D_{P}$. Let $W\in Z_{d}^{P(d)}(S)$. If $\log{W}=(c_{0}, \ldots , c_{r})$, then $\max_{1\leq i\leq r}c_{i}=dP(d)-e(d)$. There is a permutation matrix $q\in\textup{GL}_{r+1}(k)$ such that $S(q.W^{\star})$ is a Borel-fixed monomial ideal.
\end{lemma}
\begin{proof}
We can prove that $c_{i}\leq dP(d)-e(d)$ for all $0\leq i\leq r$ using Lemma~\ref{maximality}. Let us apply an action on $W$ by a permutation matrix to assume that $c_{i}\geq c_{i+1}$ for every $0\leq i\leq r-1$, if necessary. We claim that if $c_{0} \leq dP(d)-e(d)-1$ then $W\notin Z^{P(d)}_{d}(S)$. If $c_{0}=dP(d)-e(d)$ then $|\Delta_{W}|^{2}$ is at least $(dP(d)-e(d))^{2}+\frac{1}{r}[e(d)]^{2}$ by the convexity of the square sum function defined on a simplex defined by the equation $\sum_{1\leq i\leq r}c_{i}=e(d)$. Therefore it suffices to show that
\begin{displaymath}
\begin{split}
\sum_{i=0}^{r}c_{i}^{2}&\leq c_{0}^{2}+(dP(d)-c_{0})^{2}= 2c_{0}^{2}-2dP(d)c_{0}+d^{2}[P(d)]^{2}\\
&<(dP(d)-e(d))^{2} +\frac{1}{r}[e(d)]^{2}
\end{split}
\end{displaymath}
if $\frac{dP(d)}{r+1}\leq c_{0}\leq dP(d)-e(d)-1$. The inequality $c_{0}\geq \frac{dP(d)}{r+1}$ holds under the condition $\sum_{0\leq i\leq r}c_{i}=dP(d)$ because of the pigeonhole principle. The first inequality is trivial. Being equivalent to a quadratic inequality in $c_{0}$ whose discriminant is $\Phi(d)$ as mentioned before Lemma~\ref{truncate1}, the second inequality is equivalent to
\begin{displaymath}
\frac{dP(d)-\sqrt{\Phi(d)}}{2}<c_{0}<\frac{dP(d)+\sqrt{\Phi(d)}}{2}
\end{displaymath}
by \eqref{discriminant1}. The equation \eqref{lowerbound1} means that
\begin{displaymath}
\frac{dP(d)-\sqrt{\Phi(d)}}{2}<\frac{dP(d)}{r+1}.
\end{displaymath}
Since $r\geq 1$, we get
\begin{displaymath}
\begin{split}
\frac{dP(d)+\sqrt{\Phi(d)}}{2} &\leq \frac{dP(d)+\sqrt{d^{2}[P(d)]^{2}-4dP(d)e(d)+4[e(d)]^{2}}}{2}\\
&=dP(d)-e(d).
\end{split}
\end{displaymath}
Using \eqref{upperbound2}, we see that the difference between the both sides of the preceding inequality is sufficiently small; that is,
\begin{displaymath}
\bigg|dP(d)-e(d)-\frac{dP(d)+\sqrt{\Phi(d)}}{2}\bigg|<1
\end{displaymath}
so that 
\begin{displaymath}
dP(d)-e(d)-1<\frac{dP(d)+\sqrt{\Phi(d)}}{2} \leq dP(d)-e(d).
\end{displaymath}
Thus, the first statement is true. There is a permutation matrix $q\in\textup{GL}_{r+1}(k)$  such that $\log{q.W^{\star}}=(c_{0}', \ldots, c_{r}')$ satisfies $c_{i}'\geq c_{i+1}'$ for all $0\leq i\leq r-1$. If $S(q.W^{\star})$ is not Borel-fixed, then there are $n_{1}\in N(q.W^{\star})$ and $n_{2}\in S_{d}\setminus N(q.W^{\star})$ such that $x_{i}n_{1}=x_{j}n_{2}$ for some $0\leq i < j \leq r$. If $W'\in\textup{Gr}(S_{d}, Q_{P}(d))$ is generated by $N(W)\cup\{n_{2}\}\setminus\{n_{1}\}$, then 
\begin{displaymath}
|\Delta_{W'}|^{2}-|\Delta_{W^{\star}}|^{2}=(c_{i}'+1)^{2}+(c_{j}'-1)^{2}-(c_{i}')^{2}-(c_{j}')^{2}>0,
\end{displaymath}
which contradicts the maximality of $|\Delta_{W}|$ by Lemma~\ref{duality}.
\end{proof}
There are some properties satisfied by an arbitrary worst unstable point of $\textup{Gr}(S_{d}, P(d))$ for an arbitrary choice of $P\in \mathbb{Q}[t]$, which are weaker than the properties stated in Lemma~\ref{opt1}.
\begin{lemma}
\label{general1}
If $d\geq D_{P}$ then every $W\in Z_{d}^{P(d)}(S)$ satisfies the following properties.
\begin{itemize}
\item $\log{W}=(c_{0}, \ldots, c_{r})$ satisfies $dP(d)-2e(d) <\max_{0\leq i\leq r}c_{i}\leq dP(d)-e(d)$.
\item $W$ is a $P(d)$-dimensional subspace of $S_{d}$ generated by monomials and $x_{\beta}^{\lfloor \frac{d}{2} \rfloor}$ divides any monomial in $W$ where $c_{\beta}=\max_{0\leq i\leq r}c_{i}$.
\item For any monomial $n\in W$, $U_{\beta}(n)\subset W$.
\item There is a permutation matrix $q\in\textup{GL}_{r+1}(k)$ such that $S(q.W^{\star})$ is a Borel-fixed ideal.
\end{itemize}
\end{lemma}
\begin{proof}
It is possible to modify the proof of Lemma~\ref{opt1} in order to prove that the first bullet is true for any $d\geq D_{P}$; indeed, $\frac{dP(d)}{r+1}\leq c_{0}$ if we assume that $c_{0}=\max_{0\leq i\leq r}c_{i}$ after using an action by a permutation matrix. $c_{0}\leq dP(d)-2e(d)$ implies that
\begin{displaymath}
\frac{dP(d)-\sqrt{\Phi(d)}}{2}<\frac{dP(d)}{r+1}\leq c_{0}\quad (\because \eqref{lowerbound1}, \eqref{discriminant1})
\end{displaymath}
\begin{displaymath}
\leq dP(d)-2e(d)<\frac{dP(d)+\sqrt{\Phi(d)}}{2}\quad (\because \eqref{upperbound1}, \eqref{discriminant1})
\end{displaymath}
so that
\begin{displaymath}
\begin{split}
\sum_{i=0}^{r}c_{i}^{2}&\leq c_{0}^{2}+(dP(d)-c_{0})^{2}\\
&= 2c_{0}^{2}-2dP(d)c_{0}+d^{2}[P(d)]^{2}\\
&<(dP(d)-e(d))^{2} +\frac{1}{r}[e(d)]^{2}\\
&\leq \Vert\log{L(P(d), d, S)}\Vert_{\mathbb{R}}^{2},
\end{split}
\end{displaymath}
which contradicts the maximality of $\Vert\log{W}\Vert_{\mathbb{R}}$. If there is a $W\in Z_{d}^{P(d)}(S)$ which does not satisfy the third bullet, then there's a monomial $n\in W$ satisfying $U_{\beta}(n)\not\subset W$. Choose a monomial $m\in U_{\beta}(n)\setminus W$. For the $P(d)$-dimensional subspace $W'$ of $S_{d}$ spanned by $N(W)\cup \{m\}\setminus \{n\}$, there is a sequence $\{a_{i}\}_{i=0}^{r}$ of integers satisfying 
\begin{displaymath}
\bigg\langle \log{W'}, \lambda_{i}\otimes_{\mathbb{Z}}1 \bigg\rangle_{\mathbb{R}} =c_{i}+a_{i}\quad\textrm{for all }0\leq i\leq r
\end{displaymath}
so that $a_{\beta}>0$, $a_{i}\leq 0$ for all $i\neq \beta$ and $\sum_{i=0}^{r}a_{i}=0$ by the definition of $U_{\beta}(n)$. From these hypotheses, we can derive an inequality
\begin{displaymath}
|\Delta_{W'}|^{2}-|\Delta_{W}|^{2}=\sum_{i=0}^{r} 2c_{i}a_{i} + a_{i}^{2}\geq \sum_{i=0}^{r} 2c_{\beta}a_{i} + a_{i}^{2} \geq a_{\beta}^{2} >0
\end{displaymath}
by the maximality of $c_{\beta}$, which leads us to a contradiction. Therefore, the third bullet holds for $d\geq D_{P}$. If there is $W\in Z_{d}^{P(d)}(S)$ which does not satisfy the second bullet, there are monomials $m'\in M_{d}\setminus W$ and $n'\in W$ such that $x_{\beta}^{d-l(d)}\mid m'$ and $x_{\beta}^{\lfloor\frac{d}{2}\rfloor}\nmid n'$. Consider $W''$ spanned by $N(W)\cup\{m'\}\setminus \{n'\}$. There is a sequence  $\{b_{i}\}_{i=0}^{r}$ of integers satisfying 
\begin{displaymath}
\bigg\langle \log{W''}, \lambda_{i}\otimes_{\mathbb{Z}}1 \bigg\rangle_{\mathbb{R}} =c_{i}+b_{i}\quad\textrm{for all }0\leq i\leq r.
\end{displaymath}
By the defining property of $m'$ and \eqref{asymptotic1}, $b_{\beta}>\lceil \frac{d}{2} \rceil -l(d)>\frac{d}{3}$ and $b_{i}\geq -d$ for all $i\neq \beta$. Also, $\sum_{0\leq i\leq r, i\neq \beta}c_{i}<2e(d)$ by the first bullet we have proved. We see that 
\begin{displaymath}
\begin{split}
|\Delta_{W''}|^{2}-|\Delta_{W}|^{2}&=\sum_{i=0}^{r} 2c_{i}b_{i} + b_{i}^{2}\\
&\geq 2c_{\beta}b_{\beta}+ \sum_{0\leq i\leq r,i\neq \beta} 2c_{i}b_{i}\\
&>2(dP(d)-2e(d))\cdot \frac{d}{3}-4de(d)\\
&=2(dP(d)-8e(d))\cdot \frac{d}{3}>0
\end{split}
\end{displaymath}
by \eqref{asymptotic1}; this leads us to another contradiction so that the second bullet holds for all $d\geq D_{P}$. We can prove the fourth bullet as we did in the proof of Lemma~\ref{opt1}.
\end{proof}

\section{Worst Unstable Hilbert Points for a Constant Hilbert Polynomial}
We would like to investigate whether Lemma~\ref{opt1} and Lemma~\ref{general1} are useful for finding worst unstable points (that is, whether, we have $Z_{d}^{Q_{P}(d)}(S)=X_{d}^{P}(S)$ for all $d\geq g_{P}$). In this section, we compute the Hilbert polynomial of an arbitrary $W\in Z_{d}^{P(d)}(S)$ to answer this question.
\begin{lemma}
\label{sharpness}
Suppose $W\in R(d, Q_{P}(d), S)$. Assume that $W^{\star}\in R(d, P(d), S)$ satisfies all bullets in Lemma~\ref{general1} and $d\geq D_{P}$. Let $J$ be the graded ideal of $S$ generated by $W$. Then $\dim_{k}(S/SW)_{t}=P(d)$ for all $t\geq d$. In particular, If $W^{\star}\in Z_{d}^{P(d)}(S)$ then the Hilbert polynomial of $S/SW$ is constant, whose value at an arbitrary integer is equal to $P(d)$.
\end{lemma}
\begin{proof}
The set $L=N(W^{\star})$ satisfies
$1\leq\min_{m\in L}\deg_{x_{0}}m$ by the second bullet of Lemma~\ref{general1} and
\begin{displaymath}
L=\bigcup_{m\in L} U_{0}(m)=\bigcup_{m\in L}\{n|n\in k[x_{1}, \ldots , x_{r}], n \textrm{ divides } m\upharpoonright_{x_{0}=1}\}\cap M_{d}
\end{displaymath}
by the third bullet of Lemma~\ref{general1}. By definition, $J=S(M_{d}\setminus N(W^{\star}))=S(M_{d}\setminus L)$. We claim that the Hilbert polynomial of $S/J$ is constant. We have the inclusion
\begin{displaymath}
 M_{d+1}\setminus J_{d+1}=M_{d+1}\setminus (S_{1}(M_{d}\setminus L))=M_{d+1}\setminus M_{1}\bigg[\bigcap_{m\in L}M_{d}\setminus U_{0}(m)\bigg]
\end{displaymath}
\begin{displaymath}
\supset M_{d+1}\setminus \bigg[\bigcap_{m\in L} M_{1}(M_{d}\setminus U_{0}(m))\bigg].
\end{displaymath}
It's easy to check that for every $m\in M_{d}$, $M_{1}(M_{d}\setminus U_{0}(m))=M_{d+1}\setminus U_{0}(x_{0}m)$ if $x_{0}|m$. Note that $x_{0}|m$ for every $m\in L$; therefore,
\begin{displaymath}
M_{d+1}\setminus J_{d+1}\supset \bigcup_{m\in L} U_{0}(x_{0}m).
\end{displaymath}
If there is a monomial $n\in M_{d+1}$ such that $n\notin \bigcup_{m\in L} U_{0}(x_{0}m)$ and $n\notin J_{d+1}$, then $\deg_{x_{0}}n=0$; otherwise, $\frac{n}{x_{0}} \notin \bigcup_{m\in L} U_{0}(m)$ so that $n/x_{0}\in J_{d}$, which implies that $n\in J_{d+1}$. Such a conclusion contradicts our assumption $n\notin J_{d+1}$. Note that $\deg m|_{x_{0}=1}\leq d-1$ for every $m\in L$ by the second bullet of Lemma~\ref{general1} so that $n/x_{j}\in J_{d}$ for every $x_{j}$ dividing $n$ because $\deg{n/x_{j}|_{x_{0}=1}}=d>\deg m|_{x_{0}=1}$; this leads us to another contradiction against the assumption $n\notin J_{d+1}$. Therefore,
\begin{displaymath}
M_{d+1}\setminus J_{d+1} = \bigcup_{m\in L} U_{0}(x_{0}m).
\end{displaymath}
Note that this set and $L$ have the same cardinality. We can dehomogenize both sides of the preceding equation by the first variable $x_{0}$ to compare their cardinality; that is,
\begin{displaymath}
\begin{split}
&\left|\bigcup_{m\in L} U_{1}(x_{1}m)\right|\\
&=\left| \bigcup_{m\in L}\{nx_{0}^{d+1-\deg n}|n\in k[x_{1}, \ldots , x_{r}], n \textrm{ divides } m\upharpoonright_{x_{0}=1}\}\cap M_{d+1}\right|\\
&=\left|\{nx_{0}^{d+1-\deg n}|n\in k[x_{1}, \ldots , x_{r}], n \textrm{ divides } m|_{x_{0}=1}\textrm{ for some }m\in L\}\cap M_{d+1} \right|\\
&=\left|\{nx_{0}^{d-\deg n}|n\in k[x_{1}, \ldots , x_{r}], n \textrm{ divides } m|_{x_{0}=1}\textrm{ for some }m\in L\}\cap M_{d} \right|\\
&=|L|.
\end{split}
\end{displaymath}
Thus 
\begin{displaymath}
\dim_{k}(S/J)_{d+1}=\dim_{k}(S/J)_{d}=P(d).
\end{displaymath}
Using the same argument, we can show that
\begin{displaymath}
M_{d+i+1}\setminus J_{d+i+1} = \bigcup_{m\in L} U_{0}(x_{0}^{i+1}m)
\end{displaymath}
if
\begin{displaymath}
M_{d+i}\setminus J_{d+i} = \bigcup_{m\in L} U_{0}(x_{0}^{i}m)
\end{displaymath}
and thus $\dim_{k}(S/J)_{d+i+1}=\dim_{k}(S/J)_{d+i}$ for all $i\in\mathbb{N}$. By induction, $\dim_{k}(S/J)_{t}=P(d)$ for all $t\geq d$.
\end{proof}
Now we explain when Lemma~\ref{general1} is useful to find worst unstable Hilbert points.
\begin{theorem}
\label{useless}
Suppose that $d\geq D_{P}$. Let us fix a Pl\"ucker embedding of $\textup{Hilb}^{P}(\mathbb{P}^{r}_{k})$ corresponding to $d$. The following statements are equivalent.
\begin{itemize}
\item[i)] A worst unstable point of $\textup{Gr}(Q_{P}(d), S_{d})$ is in $\textup{Hilb}^{P}(\mathbb{P}^{r}_{k})$.
\item[ii)] Every worst unstable point of $\textup{Gr}(Q_{P}(d), S_{d})$ is in $\textup{Hilb}^{P}(\mathbb{P}^{r}_{k})$.
\item[iii)] $P$ is a constant polynomial.
\end{itemize}
\end{theorem}
\begin{proof}
ii)$\Longrightarrow$i) is trivial. Suppose $P$ is a constant polynomial and $W\in Z_{d}^{Q_{P}(d)}(S)$. Lemma~\ref{sharpness} and Lemma~\ref{general1} imply that $W\in\textup{Hilb}^{P}(\mathbb{P}^{r}_{k})$ under the Pl\"ucker embedding corresponding to $d\geq D_{P}\geq g_{P}$. Theorem~\ref{foundation} implies iii)$\Longrightarrow$ii). If i) is true, there is $W\in Z_{d}^{Q_{P}(d)}(S)$ such that $W\in\textup{Hilb}^{P}(\mathbb{P}^{r}_{k})$. Then the Hilbert polynomial of $S/SW$ is equal to $P$ by the definition.  By Lemma~\ref{duality}, $W^{\star}\in Z^{P(d)}_{d}(S)$ so that the Hilbert polynomial of $S/SW$ is constant so that $P$ is a constant polynomial by Lemma~\ref{sharpness}. As a result, i)$\Longrightarrow$iii) is true.
\end{proof}
Consequently, we cannot use both Lemma~\ref{general1} and Lemma~\ref{opt1} to describe worst unstable points of $\textup{Hilb}^{P}(\mathbb{P}^{r}_{k})$ when $P$ is non-constant while Lemma~\ref{opt1} is still useful when $P$ is constant.
\begin{definition}
Let $S'=k[x_{1}, \ldots, x_{r}]$. Define $\rho:\mathbb{N}\rightarrow\mathbb{N}$
\begin{displaymath}
\rho(d)=\binom{r+l(d)}{r}-P(d)=Q_{P}(d)+\binom{r+l(d)}{r}-\binom{r+d}{r}
\end{displaymath}
for all $d\geq g_{P}$. Note that $\rho$ is constant for $d\geq g_{P}$ if $P$ is constant.
\end{definition}
\begin{theorem}
\label{constant}
Suppose that $P$ is a constant polynomial and $d\geq D_{P}$. For an arbitrary worst unstable point $W$ of $\textup{Hilb}^{P}(\mathbb{P}^{r}_{k})$ with respect to $d$, there are $g\in \textup{GL}_{r+1}(k)$ and $W'\in Z_{l(d)}^{\rho(d)}(S')$ such that
\begin{displaymath}
g.W=x_{0}^{d-l(d)}W'+A\left(\binom{r+l(d)}{r}, d, S\right)
\end{displaymath}
and $S(g.W)$ is a Borel-fixed ideal. In particular, the algebraic equation defined by the graded ideal $SW$ has a unique solution in $\mathbb{P}_{k}^{r}$.
\end{theorem}
\begin{proof}
By Theorem~\ref{foundation}, there is $g_{0}\in \textup{GL}_{r+1}(k)$ such that $g_{0}.W\in Z^{Q_{P}(d)}_{d}(S)$. By Lemma~\ref{opt1} and Theorem~\ref{useless}, there is a permutation matrix $q\in \textup{GL}_{r+1}(k)$ such that $qg_{0}.W$ is Borel-fixed and $\langle \log (qg_{0}.W)^{\star}, \lambda_{0}\otimes_{\mathbb{Z}}1\rangle_{\mathbb{R}}=dP(d)-e(d)$. Let $g=qg_{0}$. By Lemma~\ref{maximality}, 
\begin{displaymath}
L\left(\binom{r+l(d)-1}{r}, d, S\right)\subset (g.W)^{\star}\subset L\left(\binom{r+l(d)}{r}, d, S\right).
\end{displaymath}
Applying $^{\star}$ to each subspace of $S_{d}$ in the preceding equation, we have
\begin{displaymath}
A\left(\binom{r+l(d)-1}{r}, d, S\right)\supset g.W \supset A\left(\binom{r+l(d)}{r}, d, S\right).
\end{displaymath}
By the construction of $A(t, d, S)$, $\deg_{x_{0}}{m}=d-l(d)$ for all $m\in N(g.W)\setminus N\left(A\left(\binom{r+l(d)}{r}, d, S\right)\right)$. Therefore, there is $W'\in R(l(d), \rho(d), S')$ such that 
\begin{displaymath}
g.W=x_{0}^{d-l(d)}W'+A\left(\binom{r+l(d)}{r}, d, S\right).
\end{displaymath}
We can prove that $W'\in Z_{l(d)}^{\rho(d)}(S)$ using the maximality of $\Vert\log{g.W}\Vert_{\mathbb{R}}$; indeed, we can express each component $\langle \log{g.W}, \chi_{i}\otimes_{\mathbb{Z}}1\rangle_{\mathbb{R}}$ of $\log{g.W}$ in terms of $\langle \log{W'}, \chi_{i}\otimes_{\mathbb{Z}}1\rangle_{\mathbb{R}}$ for all $0\leq i\leq r$ as follows:
\begin{displaymath}
\begin{split}
\Vert\log{g.W}\Vert_{\mathbb{R}}^{2}&=\left\Vert\log{x_{0}^{d-l(d)}W'}+\log{A\left(\binom{r+l(d)}{r}, d, S\right)} \right\Vert_{\mathbb{R}}^{2}\\
&=\left[\rho(d)(d-l(d))+\sum_{j=l+1}^{d}\binom{r+j-1}{r-1}(d-j)\right]^{2}\\
&\quad +\sum_{i=1}^{r}\left[\langle \log{W'}, \chi_{i}\otimes_{\mathbb{Z}}1\rangle_{\mathbb{R}}+\frac{1}{r}\sum_{j=l+1}^{d}\binom{r+j-1}{r-1}j\right]^{2}\\
&=\Vert\log{W'}\Vert_{\mathbb{R}}^{2}+\left[\rho(d)(d-l(d))+\sum_{j=l+1}^{d}\binom{r+j-1}{r-1}(d-j)\right]^{2}\\
&\quad +\frac{1}{r}\left[\sum_{j=l+1}^{d}\binom{r+j-1}{r-1}j\right]^{2}+\frac{2l(d)\rho(d)}{r}\sum_{j=l+1}^{d}\binom{r+j-1}{r-1}j.
\end{split}
\end{displaymath}
Let's prove the last sentence of the theorem. Lemma~\ref{general1} implies that $x_{0}$ divides every monomial in $W\in Z_{d}^{P(d)}(S)$ if $d\geq D_{P}$, so that $x_{i}^{d}\in A\left(\binom{r+l(d)}{r}, d, S\right)$ for all $1\leq i\leq r$. Therefore, the ideal generated by $W$ has a unique solution $[1:0:\ldots : 0]$.
\end{proof}
 As mentioned in the introduction, Theorem~\ref{constant} means that every zero dimensional projective scheme which is represented by a worst unstable Hilbert point has a unique closed point. If we set $r=2$, we get the following consequence.
\begin{corollary}
\label{planepoint}
Let $c$ be a positive integer, which we view as a constant Hilbert polynomial. For every worst unstable point $x$ of $\textup{Hilb}^{c}(\mathbb{P}^{2}_{k})$, there is $h\in\textup{GL}_{3}(k)$ such that $h.x$ represents a projective scheme corresponding to the ideal generated by $A\left(c, g_{c}, S\right)$. In particular, the solution of the algebraic equation defined by the ideal of $S$ generated by $A\left(c, g_{c}, S\right)$ is unique so that the projective scheme represented by $x$ has a unique closed point.
\end{corollary}
\begin{proof}
We see that $h. x=A\left(c, d, S\right)$ for some $h\in\textup{GL}_{3}(k)$ if $x$ is a worst unstable point of $\textup{Hilb}^{c}(\mathbb{P}^{2}_{k})$ with respect to $d\geq D_{P}$ by Theorem~\ref{constant}. We also see that $A\left(c, g_{c}, S\right)$ and $A\left(c, d, S\right)$ define the same projective scheme by \cite[(1.2)(iii)]{Gotzmann}.
\end{proof}

\section{An Application of Murai's Result}
We described the form of almost all monomials in a point $W\in Z_{d}^{Q_{P}(d)}(S)$ when the Hilbert polynomial $P$ is constant in Theorem~\ref{constant}. For non-constant Hilbert polynomials, however, we have to make use of the \textit{minimal growth condition} of Hilbert points described in  \cite[(1) on page 844]{Murai}.
\begin{theorem}
\label{murai}
Suppose $M\subset M_{d}$ generates an ideal $I$. Let $P$ and $Q_{P}$ be the Hilbert polynomials of $S/I$ and $I$, respectively. Assume $d\geq g_{P}$. If $[\wedge^{Q_{P}(d)}I_{d}]\in \textup{Hilb}^{P}(\mathbb{P}^{r}_{k})$ and $\binom{r+\delta(d) -1 }{r}< Q_{P}(d)\leq\binom{r+\delta(d) }{r}$ for some $0\leq\delta(d)\leq d$ then there is a monomial $m$ of degree $d-\delta(d)$ such that $m$ divides any monomial in $M$ and there is $0\leq i\leq r$ satisfying $U_{i}(n)\subset M_{d}\setminus M$ for all $n\in M_{d}\setminus M$. The converse of the preceding sentence is true when $r=2$.
\end{theorem}
 The preceding theorem is \cite[Proposition 8]{Murai}. Although Lemma~\ref{general1} cannot be used to describe worst unstable points for general $P$ by Theorem~\ref{useless}, such a fact means that describing the {\it shape} of a general Gotzmann set of monomials \cite{Murai} helps us to find general worst unstable Hilbert points. A property of an arbitrary Hilbert point which is a wedge of monomials has been described in Theorem~\ref{murai}. In this section, we modify Theorem~\ref{sharpness} to describe general worst unstable Hilbert points for some non-constant $P$. 
\begin{definition}
Let us introduce some notation. Put $p(d)=\binom{r+\delta(d)}{r}-Q_{P}(d)$ and $\alpha(d)=P(d)-p(d)=\binom{r+d}{r}-\binom{r+\delta(d)}{r}$. The polynomial $Q_{P}$ is the Hilbert polynomial of the ideal $S(L(Q_{P}(d), d, S))$ for all $d\geq g_{P}$, as explained in \cite{Gotzmann}. By the construction of $\delta$, there is a relation $x_{0}^{d-\delta(d)}x_{r}^{\delta(d)}\leq_{\textup{lex}} \mu(Q_{P}(d), d)<_{\textup{lex}} x_{0}^{d-\delta(d)+1}x_{r}^{\delta(d)-1}$. The monomial $\nu=\mu(Q_{P}(d), d)\upharpoonright_{x_{0}=1}\in S'$ is the $Q_{P}(d)-\binom{r+\delta(d)-1}{r}$'th largest monomial in $S'_{\delta(d)}$, with respect to the lexicographic order of monomials in $S'$ induced by the inclusion $S'\subset S$. Thus, for any $t\geq 0$, 
\begin{displaymath}
\begin{split}
\dim_{k} &\left(S'L\left(Q_{P}(d)-\binom{r+\delta(d)-1}{r}, \delta(d), S'\right)_{\delta(d+t)}\right)\\
&=|\{m\in S'_{\delta(d+t)}|m\textrm{ is a monomial and } m\geq_{\textup{lex}} x_{r}^{t}\nu \}|\\
&=\left|\{m\in M_{d+t}| x_{r}^{t}\mu(Q_{P}(d), d)\leq_{\textup{lex}}m<_{\textup{lex}} x_{0}^{d-\delta(d)+1}x_{r}^{\delta(d+t)-1}\}\right|\\
&=Q_{P}(d+t)-\binom{r+\delta(d+t)-1}{r}\\
&=\binom{r+\delta(d+t)-1}{r-1}-p(d+t).
\end{split}
\end{displaymath}
Therefore, $\binom{r+t'-1}{r-1}-p(t'+\gamma)$ is the Hilbert polynomial of the ideal $J$ of $S'$ generated by $L\left(Q_{P}(d)-\binom{r+\delta(d)-1}{r}, \delta(d), S'\right)$. 
Since $p(t+\gamma)\in\mathbb{Q}[t]$ is a Hilbert polynomial, we can define $l$ and $e$ for $p(t+\gamma)$ in $S'$ as we did for $P$ in $S$ before. Let us denote these functions by $l'$ and $e'$ respectively. That is, $l'(d)$ is the integer satisfying
\begin{displaymath}
\sum_{0\leq i\leq l'(d)-1}{\binom{r+i-2}{r-2}} < p(d+\gamma) \leq \sum_{0\leq i\leq l'(d)}{\binom{r+i-2}{r-2}}
\end{displaymath}
and
\begin{displaymath}
e'(d)=\bigg[p(d+\gamma)-\sum_{0\leq i\leq l'(d)-1}{\binom{r+i-2}{r-2}}\bigg]l'(d)+\sum_{0\leq i\leq l'(d)-1}{\binom{r+i-2}{r-2}}i.
\end{displaymath}
On the other hand, define
\begin{displaymath}
\epsilon(d)=\sum_{i=\delta(d)+1}^{d}{\binom{r+i-1}{r-1}i}.
\end{displaymath}
We can easily see that
\begin{displaymath}
\log{A\left(\binom{r+\delta(d)}{r},d, S\right)}=\bigg(d\alpha(d) -\epsilon(d), \frac{\epsilon(d)}{r}, \ldots, \frac{\epsilon(d)}{r} \bigg).
\end{displaymath}
If $\gamma =0$ then $P=p$ and $\alpha=0=\epsilon$. Also, worst unstable points of $\textup{Gr}(Q_{P}(d), S_{d})$ satisfy the conclusion of Theorem~\ref{murai}, as we have seen in Lemma~\ref{general1}. Let's assume that $\gamma\geq 1$. Define the second \textit{discriminant} $\Omega:\mathbb{N}\rightarrow\mathbb{Z}$ and the \textit{sum} $C:\mathbb{N}\rightarrow\mathbb{Z}$ as follows:
\begin{displaymath}
\Omega(d)=\bigg(p(d)\gamma +d\alpha(d) -\epsilon(d) -\frac{\epsilon(d)}{r} -p(d)\delta(d)\bigg)^{2}-8p(d)\delta(d)e'(d) +8[e'(d)]^{2}
\end{displaymath}
and
\begin{displaymath}
C(d)=-p(d)\gamma-d\alpha(d)+\epsilon(d)+\frac{\epsilon(d)}{r}+p(d)\delta(d).
\end{displaymath}
\end{definition}

In Theorem~\ref{opt2}, we will see that $\Omega$ and $C$ are actually a discriminant and the sum of the roots of some quadratic equation, respectively. These functions have the following asymptotic behavior.
\begin{lemma}
\label{truncate2}
If $\gamma\neq 0$ then 
\begin{equation}
\lim_{d\rightarrow\infty}{\frac{p(d)\delta(d)}{C(d)+\sqrt{\Omega(d)}}}=0,
\end{equation}
\begin{equation}
\lim_{d\rightarrow\infty}{\frac{p(d)}{P(d)}}=0
\end{equation}
and
\begin{equation}
\label{e2}
\lim_{d\rightarrow\infty}\frac{e'(d)}{dp(d+\gamma)}=0.
\end{equation}
Furthermore, if $p$ is a constant,
\begin{equation}
\lim_{d\rightarrow\infty}\frac{C(d)-\sqrt{\Omega(d)}}{2}=0^{+}.
\end{equation}
\end{lemma}
\begin{proof}
Note that (\ref{e2}) follows from the proof of lemma~\ref{truncate1}. \\It suffices to show that

\begin{displaymath}
\lim_{d\rightarrow\infty}\frac{\alpha(d)}{P(d)}=1 
\end{displaymath}
and 
\begin{displaymath}
\lim_{d\rightarrow\infty}\frac{\epsilon(d)}{dP(d)}=1.
\end{displaymath}
Recall that there is a lex-segment ideal $J$ generated by $J_{g_{P}}$, whose Hilbert polynomial is $Q_{P}$. Suppose that $\mu$ is the last monomial of $J_{g_{P}}$ with respect to the lexicographic order we fixed. We can write $\mu$ as follows:
\begin{displaymath}
\mu=x_{0}^{a_{0}-1}x_{1}^{a_{1}-a_{0}}\ldots x_{n}^{a_{n}-a_{n-1}+1}
\end{displaymath}
where
\begin{displaymath}
1=a_{-1}\leq a_{0}\leq a_{1}\leq\ldots\leq a_{n}=g_{P}
\end{displaymath}
and $n<r$. From the definition of $\gamma$, we can prove that $x_{0}^{\gamma}x_{r}^{d-\gamma}\leq_{\textup{lex}}x_{0}^{d-g_{P}}\mu<_{\textup{lex}}x_{0}^{\gamma+1}x_{r}^{d-\gamma-1}$. Therefore, $a_{0}=\gamma +1$. Recall that $Q_{P}$ has a Macaulay representation
\begin{displaymath}
Q_{P}(d)=\sum_{i=0}^{n}\binom{r-i+d-a_{i}}{r-i}.
\end{displaymath}

As a result, the leading term of $P(d)=\binom{r+d}{r}-Q_{P}(d)$ is
\begin{equation}
\label{leadt}
\frac{\gamma}{(r-1)!}d^{r-1}.
\end{equation}

Note that for all $d\geq g_{P}$,
\begin{displaymath}
\alpha(d)=\binom{r+d}{r}-\binom{r+\delta(d)}{r}=\sum_{i=1}^{\gamma}{\binom{r+d-\gamma+i-1}{r-1}}
\end{displaymath}
so that $\alpha$ and $P$ have the same leading term. This proves that $\lim_{d\rightarrow\infty}\alpha(d)/P(d)=1$. By the definition, $\epsilon(d)$ is a sum of $\gamma$ monic polynomials in $d$ of non-homogeneous degree $r$ so that $\lim_{d\rightarrow\infty}\epsilon(d)/(dP(d))=1$.
\end{proof}
\begin{definition}
Suppose that $p$ is a constant polynomial and $\gamma\neq 0$. , Let $D^{P}\in\mathbb{N}$ be the minimal integer satisfying $\Omega(d)>0$, $D^{P}\geq g_{P}$, $D^{P}\geq D_{p}+\gamma$ and
\begin{equation}
\label{region2}
0<\frac{C(d)-\sqrt{\Omega(d)}}{2}<1 < p(d)\delta(d)<\frac{C(d)+\sqrt{\Omega(d)}}{2}
\end{equation}
for all $d\geq D^{P}$. Such $D^{P}$ exists by Lemma~\ref{truncate2}. Let $b'(d)=\binom{r+\delta(d)-1}{r-1}-p$ and let $I(W')$ denote the graded ideal of $S$ generated by
\begin{displaymath}
x_{0}^{\gamma}\left(L\left(\binom{r+\delta(d)-1}{r}, \delta(d), S\right)+W'\right)
\end{displaymath}
for an arbitrary $W'\in Z_{\delta(d)}^{b'(d)}(S')$.
\end{definition}
\begin{lemma}
\label{Hpolynomial}
Suppose that $d\geq D^{P}$ and \eqref{goodsit} holds. For all $W'\in Z_{\delta(d)}^{b'(d)}(S')$, the Hilbert polynomial of $I(W')$ is $Q_{P}$.
\end{lemma}
\begin{proof}
It is clear that
\begin{equation}
\label{lexseg}
L\left(\binom{r+t-1}{r}, t, S\right)=\left\{f\in S_{t}|x_{0}\textup{ divides }f \right\}
\end{equation}
for all $t\in\mathbb{N}$ by the construction of $L(t, d, S)$. Also, 
\begin{displaymath}
I(W')_{d+1}\supset x_{0}^{\gamma}\left(L\left(\binom{r+\delta(d+1)-1}{r}, \delta(d+1), S\right)+(S'W')_{\delta(d)+1}\right)
\end{displaymath}
by the construction of $I(W')$. 

If $f\in x_{0}^{\gamma}L\left(\binom{r+\delta(d)-1}{r}, \delta(d), S\right)$, then  $x_{j}f\in x_{0}^{\gamma}L\left(\binom{r+\delta(d+1)-1}{r}, \delta(d+1), S\right)$ for all $0\leq j\leq r$ by \eqref{lexseg}. If $f\in x_{0}^{\gamma}W'$, then 
$x_{0}f\in x_{0}^{\gamma}L\left(\binom{r+\delta(d+1)-1}{r}, \delta(d+1), S\right)$ by \eqref{lexseg} and $x_{j}f\in (S'W')_{d+1}$ for all $1\leq j\leq r$ by the definition. Therefore,
\begin{displaymath}
I(W')_{d+1}= x_{0}^{\gamma}\left(L\left(\binom{r+\delta(d+1)-1}{r}, \delta(d+1), S\right)+(S'W')_{\delta(d)+1}\right).
\end{displaymath}
By \eqref{lexseg}, we see that
\begin{displaymath}
L\left(\binom{r+\delta(d+1)-1}{r}, \delta(d+1), S\right)\cap(S'W')_{\delta(d)+1}=\{0\}
\end{displaymath}
and
\begin{displaymath}
\dim_{k} L\left(\binom{r+\delta(d+1)-1}{r}, \delta(d+1), S\right)=\binom{r+\delta(d+1)-1}{r}.
\end{displaymath}
By Lemma~\ref{sharpness}, $\dim_{k}(S'W')_{\delta(d)+1}=b'(d+1)$ since $\delta(d)\geq D_{p}$ by the defining property of $D^{P}$. Therefore, 
\begin{displaymath}
\begin{split}
\dim_{k} I(W')_{d+1}&=\dim_{k}(S'W')_{\delta(d)+1}+\dim_{k} L\left(\binom{r+\delta(d+1)-1}{r}, \delta(d+1), S\right)\\
&=b'(d+1)+\binom{r+\delta(d+1)-1}{r}\\
&=\binom{r+\delta(d+1)-1}{r-1}+\binom{r+\delta(d+1)-1}{r}-p\\
&=\binom{r+\delta(d+1)}{r}-p\\
&=Q_{P}(d+1)
\end{split}
\end{displaymath}
by the construction of $p$. Similarly, we can show that $\dim_{k}I(W')_{d}=Q_{P}(d)$ by the definition of $I(W')$ and \eqref{lexseg}. Now we can use the Gotzmann persistence theorem \cite[Satz on p.61]{Gotzmann} to show that the Hilbert polynomial of $I(W')$ is equal to $Q_{P}$.
\end{proof}

\begin{theorem}
\label{opt2}
Suppose that the Hilbert polynomial $P$ satisfies \eqref{goodsit} for some $\gamma\neq 0$ and a natural number $p$. If $W\in \textup{Gr}(S_{d}, Q_{P}(d))$ is a worst unstable point of $\textup{Hilb}^{P}(\mathbb{P}^{r}_{k})$ with respect to $d\geq D^{P}$, then there are $g\in \textup{GL}_{r+1}(k)$ and $W'\in Z_{\delta(d)}^{b'(d)}(S')$ such that
\begin{displaymath}
g.W=I(W')
\end{displaymath}
and the ideal $S(g.W)$ of $S$ is Borel-fixed.
\end{theorem}
\begin{proof}
There is $g\in \textup{GL}_{r+1}(k)$ such that $g.W\in X_{d}^{P}(S)$ by Theorem~\ref{foundation}. Let $\log{g.W}=(c_{0}', c_{1}', \ldots, c_{r}')\in X(T)_{\mathbb{R}}$. After $g.W$ is acted by a suitable permutation matrix, we can assume that $c_{i}'\geq c_{i+1}'$ for all $0\leq i \leq r-1$. By Theorem~\ref{murai}, there is $m\in M_{\gamma}$ such that $m$ divides every monomial in $g.W$ because $g.W\in\textup{Hilb}^{P}(\mathbb{P}^{r}_{k})$. There is $V\in R_{\delta(d), Q_{P}(d)}(S)$ such that $g.W=mV$. If $m\neq x_{0}^{\gamma}$ then $|\Delta_{g.W}|<|\Delta_{x_{0}^{\gamma}V}|$ by the maximality of $c_{0}'$. Therefore $m= x_{0}^{\gamma}$. Let $\log{V^\star}=(c_{0}, c_{1}, \ldots, c_{r})$. We can derive a relation between $\log{g.W^\star}$ and $\log{V^\star}$ as follows:
\begin{displaymath}
\begin{split}
&\log{g.W^\star}\\
&=\log\left[ g.W^{\star}\cap  L\left(\binom{r+\delta(d)}{r},d, S\right)\right] +\log\left[ g.W^{\star}\cap A\left(\binom{r+\delta(d)}{r},d, S\right)\right]\\
&=\log\left[g.W^{\star}\cap x_{0}^{\gamma}S_{\delta(d)}\right]+\log {A\left(\binom{r+\delta(d)}{r},d, S\right)}\\
&=\log{x_{0}^{\gamma}V^{\star}}+\log{A\left(\binom{r+\delta(d)}{r},d, S\right)}\\
&=\log{V^\star}+\bigg(p\gamma +d\alpha(d) -\epsilon(d), \frac{\epsilon(d)}{r}, \ldots, \frac{\epsilon(d)}{r} \bigg).
\end{split}
\end{displaymath}
Thus, we can derive a relation between $\left\Vert \log{g.W^\star}\right\Vert_{\mathbb{R}}^{2}$ and $\left\Vert \log{V^\star}\right\Vert_{\mathbb{R}}^{2}$:
\begin{equation}
\label{relation}
\begin{split}
&\left\Vert \log{g.W^\star}\right\Vert_{\mathbb{R}}^{2}\\
&=\left\Vert \log{V^\star}\right\Vert_{\mathbb{R}}^{2}+2\left(p\gamma+d\alpha(d)-\epsilon(d)-\frac{\epsilon(d)}{r}\right)c_{0}\\
&\quad+2\frac{\epsilon(d)}{r}p\delta(d)+\left(p\gamma +d\alpha(d) -\epsilon(d)\right)^{2}+ \frac{[\epsilon(d)]^{2}}{r}.
\end{split}
\end{equation}
If $c_{0}=0$, then $V^{\star}\in Z^{p}_{\delta(d)}(S')$ by the maximality of $\Vert \log g.W^{\star}\Vert$ (which is given by \eqref{dual} and Lemma~\ref{Hpolynomial}) so that $W'$ mentioned in the statement exists by the construction of $I(W')$ and Lemma~\ref{duality}.
To prove that $c_{0}=0$, we can show that
\begin{displaymath}
2\bigg(\gamma p(d) +d\alpha(d) -\epsilon(d)-\frac{\epsilon(d)}{r} \bigg)c_{0}+\sum_{i=0}^{r}c_{i}^{2}-(p(d)\delta(d) -2e'(d))^{2}> 0
\end{displaymath}
first. We can prove that
\begin{displaymath}
\begin{split}
0&\leq \left\Vert \log{g.W^\star}\right\Vert_{\mathbb{R}}^{2}-\left\Vert\log\left[x_{0}^{\gamma}(W')^{\star}+A\left(\binom{r+\delta(d)}{r},d, S\right)\right]\right\Vert_{\mathbb{R}}^{2}\\
&=\sum_{i=0}^{r}c_{i}^{2}+2\bigg(p(d)\gamma +d\alpha(d) -\epsilon(d)-\frac{\epsilon(d)}{r} \bigg)c_{0}-\left\Vert\log{(W')^{\star}}\right\Vert_{\mathbb{R}}^{2}
\end{split}
\end{displaymath}
for all $W'\in Z_{\delta(d)}^{b'(d)}(S')$ from the maximality of $\Vert g.W^{\star}\Vert$ using Lemma~\ref{Hpolynomial} and that
\begin{displaymath}
\Vert\log{(W')^\star}\Vert^{2}_{\mathbb{R}}> (p(d)\delta(d)-2e'(d))^{2}
\end{displaymath}
for all $W'\in Z_{\delta(d)}^{b'(d)}(S')$ by Lemma~\ref{general1}. We have the inequality
\begin{displaymath}
\begin{split}
0&<2\bigg(p(d)\gamma +d\alpha(d) -\epsilon(d)-\frac{\epsilon(d)}{r} \bigg)c_{0}+\sum_{i=0}^{r}c_{i}^{2}-(p(d)\delta(d) -2e'(d))^{2}\\
&\leq 2c_{0}^{2}-2C(d)c_{0}+4p(d)\delta(d) e'(d)-4[e'(d)]^{2}.
\end{split}
\end{displaymath}
Being a quadratic inequality in $c_{0}$, the preceding inequality is equivalent to
\begin{displaymath}
\frac{C(d)-\sqrt{\Omega(d)}}{2}>c_{0} \quad\textrm{  or  }\quad c_{0}>\frac{C(d)+\sqrt{\Omega(d)}}{2}.
\end{displaymath}
By \eqref{region2}, the second condition is redundant because $c_{0}\leq p(d)\delta(d)$. That is, $c_{0}=0$ by \eqref{region2} again.\\
If $S(g.W)$ is not Borel-fixed, then there are $n_{1}\in N(g.W)$ and $n_{2}\in S_{d}\setminus N(g.W)$ such that $x_{i}n_{1}=x_{j}n_{2}$ for some $1\leq i< j\leq r$. The existence of $W'$, Lemma~\ref{opt1} and Lemma~\ref{duality} means that $i\neq 1$ and $j\neq r$. Therefore, consider the subspace $W''\in\textup{Gr}(S_{d}, Q_{P}(d))$ generated by $N(g.W)\cup \{n_{2}\}\setminus \{n_{1}\}$ satisfies $(SW'')_{t}=(S(g.W))_{t}$ for all $t\geq d$ by Theorem~\ref{sharpness}. That is, $W''\in \textup{Hilb}^{P}(\mathbb{P}^{r}_{k})$. Furthermore, $\Vert\log{W''}\Vert^{2}_{\mathbb{R}}>\Vert\log{g.W}\Vert^{2}_{\mathbb{R}}$ because $c_{i}\geq c_{j}$, which cannot happen by the maximality of $\Vert\log{g.W}\Vert_{\mathbb{R}}$.
\end{proof}
Theorem~\ref{constant} and Theorem~\ref{opt2} describes worst unstable Hilbert points when $d\gg 0$. Observe that these worst unstable Hilbert points remain unchanged for all but finitely many choices of $d$. We state this as a theorem.
\begin{theorem}
\label{unchanged}
If $P$ satisfies \eqref{goodsit}, then $X^{P}_{d}=X^{P}_{d'}$ for all $d, d'\gg 0$.
\end{theorem}
Compared with \cite[0.2.3]{VGIT} and \cite[Theorem 3.1.]{M3}, we can expect that this is true for an arbitrary Hesselink strata.

When $r=2$, it is easy to check that any worst unstable point of $\textup{Gr}(S'_{d}, b'(d))$ is in the $\textup{GL}_{r+1}(k)$-orbit of a lex-segment ideal. We state this as a corollary of Theorem~\ref{opt2}.
\begin{corollary}
\label{planecurve}
For an arbitrary Hilbert scheme $\textup{Hilb}^{P}(\mathbb{P}^{2}_{k})$ of 1-dimensional closed subschemes of $\mathbb{P}^{2}_{k}$, every worst unstable point of $\textup{Hilb}^{P}(\mathbb{P}^{2}_{k})$ is in $\textup{GL}_{r+1}(k)$-orbit of a lex-segment ideal.
\end{corollary}
\begin{proof}
When $r=2$, $S'=k[x_{1}, x_{2}]$ so that  
\begin{displaymath}
Z_{\delta(d)}^{b'(d)}(S')=\left\{\left[\bigwedge_{i=0}^{b'(d)-1}x_{1}^{i}x_{2}^{\delta(d)-i}\right], \left[\bigwedge_{i=0}^{b'(d)-1}x_{1}^{\delta(d)-i}x_{2}^{i}\right]\right\}
\end{displaymath}
by the definition of $Z_{\delta(d)}^{b(d)}(S')$. Whichever element of $Z_{\delta(d)}^{b(d)}(S')$ has been chosen as $W'$, $I(W')$ is equal to a lex-segment ideal after an action by a suitable permutation matrix is applied to $I(W')$. Now the claimed statement is clear by Theorem~\ref{opt2}.
\end{proof}

For all $d\geq g_{P}$, 
\begin{displaymath}
P(d)=\sum_{i=1}^{\gamma}\binom{r+d-i}{d-i+1}+\sum_{i=1}^{p}\binom{d-\gamma-i+1}{d-\gamma-i+1}
\end{displaymath}
if \eqref{goodsit} holds. This is the Macaulay representation of $P$ at $d$. Let us define
\begin{displaymath}
P_{0}(d)=\binom{r+d}{r}-\binom{r+d-\gamma}{r}=\sum_{i=1}^{\gamma}\binom{r+d-i}{d-i+1}\quad\textrm{and}\quad P_{1}=p=P-P_{0}.
\end{displaymath}

 As we have seen, $\gamma$ and $p$ can be computed from the Macaulay representation of $P$. Furthermore, we can decompose worst unstable point into other worst unstable points of {\it minor} Hilbert polynomials.
\begin{corollary}
Suppose that $P$ is in the form \eqref{goodsit} and $\gamma\neq 0$. For every worst unstable point $z$ of $\textup{Hilb}^{P}(\mathbb{P}^{r}_{k})$ and each $0\leq i\leq 1$ there is a worst unstable point $z_{i}$ of $\textup{Hilb}^{P_{i}}(\mathbb{P}^{r-i}_{k})$ satisfying the following conditions.
\begin{itemize}
\item The projective scheme represented by $z$ is a union of two closed subschemes $H$ and $c$ of $\mathbb{P}^{r}_{k}$, where $H$ is represented by $z_{0}$ and $c$ is a zero dimensional scheme.
\item The reduced scheme $H_{\textup{red}}$ associated to the scheme $H$ is isomorphic to $\mathbb{P}^{r-1}_{k}$. Furthermore, the corresponding closed immersion $H_{\textup{red}}\hookrightarrow H\hookrightarrow \mathbb{P}^{r}_{k}$ is a hyperplane embedding.
\item With above closed immersions, $c\cap H_{\textup{red}}$ is a closed subscheme of $H_{\textup{red}}\cong\mathbb{P}^{r-1}_{k}$, whose Hilbert polynomial is $P_{1}$. Also, $z_{1}$ represents $c\cap H_{\textup{red}}$.
\end{itemize}
\end{corollary}
\begin{proof}
The projective scheme represented by $x$ is defined by the ideal generated by $g.W$ in Theorem~\ref{opt2}, after applying a change of coordinates if necessary. Let
\begin{displaymath}
z_{0}=x_{0}^{\gamma}\left(L\left(\binom{r+\delta(d)}{r}, \delta(d), S\right)\right)\in R\left(d, \binom{r+\delta(d)}{r}, S\right)
\end{displaymath}
and $z_{1}=W'\in Z_{\delta(d)}^{b'(d)}(S')$. Then $H$ is defined by the graded ideal generated by $x_{0}^{\gamma}$. Let $c$ be the closed subscheme of $\mathbb{P}_{k}^{r}$ which is defined by the homogeneous ideal generated by $L\left(\binom{r+\delta(d)-1}{r}, \delta(d), S\right)+W'$. Theorem~\ref{opt2} means that $z_{i}$'s are worst unstable points. Then the remaining claims easily follow.
\end{proof}

\section{$K$-stability and Castelnuovo-Mumford regularity of worst unstable points}
Throughout this section, let's assume \eqref{goodsit}. For these Hilbert polynomials, we described worst unstable Hilbert points. In Theorem~\ref{Kinstability}, we will prove that a {\it worst unstable projective scheme} (that is, the projective scheme represented by a worst unstable Hilbert point) whose Hilbert polynomial is \eqref{goodsit} fails to be $K$-stable if $\gamma\neq 1$. Suppose that $I$ is a saturated monomial ideal of $S$ and $\lambda\in\Gamma(\textup{SL}_{r+1})$. Assume that $S/I$ has Hilbert polynomial $P$. Define
\begin{displaymath}
F_{I, \lambda}(d)=\frac{\langle\log{I_{d}^\star}, \lambda\otimes_{\mathbb{Z}}\mathbbm{1}\rangle_{\mathbb{R}}}{dP(d)}.
\end{displaymath}
Actually, $F_{I, \lambda}$ admits the expansion \cite[p.293]{Donaldson}
\begin{displaymath}
F_{I, \lambda}(d)=\sum_{i=0}^{\infty}B^{I, \lambda}_{i}\frac{1}{d^{i}}
\end{displaymath}
for sufficiently large $d$.
Let $X_{I}$ denote the projective scheme defined by the ideal sheaf $\tilde{I}$ of $\mathbb{P}^{r}_{k}$. The Donaldson-Futaki invariant of the polarized scheme $(X_{I}, \mathcal{O}_{X}(1))$ and 1-parameter subgroup $\lambda$ defined in \cite[p.294]{Donaldson} and \cite[Definition 6.]{GTian} is $B^{I, \lambda}_{1}$. 
\begin{theorem}
\label{Kinstability}
Suppose that a Hilbert polynomial $P$ satisfies \eqref{goodsit} and $\gamma\neq 1$. If $x$ is a worst unstable point of $\textup{Hilb}^{P}(\mathbb{P}_{k}^{r})$, then the projective scheme represented by $x$ is not $K$-stable.
\end{theorem}
\begin{proof}
Fix $\lambda_{r}=(r, -1, \ldots, -1)\in\Gamma(T_{r+1})\cong \mathbb{Z}^{r+1}$. If $I$ is the saturation of the ideal generated by $g.W$ in Theorem~\ref{constant}, we can compute that
\begin{displaymath}
\begin{split}
F_{I, -\lambda_{r}}(d)&=\frac{\langle\log{I_{d}^\star}, -\lambda_{r}\otimes_{\mathbb{Z}}\mathbbm{1}\rangle_{\mathbb{R}}}{dP(d)}\\
&=\frac{-rdP(d)+(r+1)e(d)}{dP(d)}\\
&=-r+\frac{(r+1)e(d)}{P(d)}\frac{1}{d}
\end{split}
\end{displaymath}
and functions $e$ and $P$ are constant for $d\geq g_{P}$ so that $B^{I, -\lambda_{r}}_{1}>0$. This means that $X_{I}$ is not $K$-stable in this case \cite[p. 294]{Donaldson}. If $I$ is the saturation of the ideal generated by $g.W$ in Theorem~\ref{opt2}, then
\begin{displaymath}
F_{I, \lambda_{r}}(d)=\frac{\langle\log{I_{d}^\star}, \lambda_{r}\otimes_{\mathbb{Z}}\mathbbm{1}\rangle_{\mathbb{R}}}{dP(d)}=\frac{rp(d)\gamma+rd\alpha(d)-(r+1)\epsilon(d)-p(d)\delta(d)}{dP(d)}.
\end{displaymath}
In the proof and the statement of Lemma~\ref{truncate2}, we have shown that
\begin{displaymath}
\lim_{d\rightarrow\infty}\frac{\alpha(d)}{P(d)}=1=\lim_{d\rightarrow\infty}\frac{\epsilon(d)}{dP(d)}
\end{displaymath}
so that
\begin{displaymath}
B_{0}^{I, \lambda_{r}}=\lim_{d\rightarrow\infty}{F_{I,\lambda_{r}}(d)}=-1
\end{displaymath}
and
\begin{displaymath}
\begin{split}
B_{1}^{I, \lambda_{r}}&=\lim_{d\rightarrow\infty}{d\left(F_{I,\lambda_{r}}(d)-B_{0}^{I, \lambda_{r}}\right)}\\
&=\lim_{d\rightarrow\infty}\frac{rp(d)\gamma+r(d\alpha-dP(d))-(r+1)(\epsilon(d)-dP(d))-p(d)\delta(d)}{P(d)}\\
&=\lim_{d\rightarrow\infty}\frac{(r+1)p(d)\gamma-(r+1)(\epsilon(d)-d\alpha(d))}{P(d)}\\
&=\lim_{d\rightarrow\infty} \frac{(r+1)p(d)\gamma}{P(d)}+\frac{r+1}{P(d)}\sum_{i=1}^{\gamma}\binom{r+i+d-\gamma-1}{r-1}(\gamma-i).
\end{split}
\end{displaymath}
Recall that the leading term of $P$ is \eqref{leadt} so that
\begin{displaymath}
B_{1}^{I, \lambda_{r}}=\frac{(r+1)(\gamma-1)}{2}> 0
\end{displaymath}
if $\gamma\neq 1$. Such a conclusion means that the projective scheme represented by $x$ is not $K$-stable.
\end{proof}
As pointed out in the introduction, it is natural to ask if there is a positive relation between Castelnuovo-Mumford regularity and GIT stability. We will simply compute the regularities of each worst unstable Hilbert points we have described in this paper and compare it with $g_{P}$, the sharp upper bound for the regularity of an arbitrary Hilbert point in $\textup{Hilb}^{P}(\mathbb{P}^{r}_{k})$ \cite{Gotzmann}.

If $p\neq 0$ and $\gamma=0$, let $l_{P}$ be an integer satisfying
\begin{displaymath}
\sum_{0\leq i\leq l_{P}-1}{\binom{r+i-1}{r-1}} < p \leq \sum_{0\leq i\leq l_{P}}{\binom{r+i-1}{r-1}}.
\end{displaymath}
If $p\neq 0$ and $\gamma\neq 0$, let $l_{P}$ be an integer satisfying
\begin{displaymath}
\sum_{0\leq i\leq l_{P}-1}{\binom{r+i-2}{r-2}} < p \leq \sum_{0\leq i\leq l_{P}}{\binom{r+i-2}{r-2}}.
\end{displaymath}
If $p=0$, let $l_{P}=-1$.
Now we are ready to compute regularity of worst unstable points in $\textup{Hilb}^{P}(\mathbb{P}^{r}_{k})$.
\begin{theorem}
\label{regularity}
Given a Hilbert polynomial $P$ satisfying \eqref{goodsit}, the Castelnuovo-Mumford regularity of the ideal sheaf on $\mathbb{P}_{k}^{r}$ corresponding to $x$ is $l_{P}+\gamma+1$ for every worst unstable point $x$ of $\textup{Hilb}^{P}(\mathbb{P}^{r}_{k})$.
\end{theorem}
\begin{proof}
Note that there is $g\in \textup{GL}_{r+1}(k)$ such that the saturated ideal represented by $g.x$ has a minimal generator of maximal degree $l_{P}+\gamma+1$ and is Borel-fixed by Theorem~\ref{opt2}, Lemma~\ref{opt1} and Lemma~\ref{duality}. By \cite[Proposition 2.11 in the part of Mark L. Green]{Elias}, the regularity of $x$ is $l_{P}+\gamma+1$.
\end{proof}
Consider the lex-segment ideal $I$ generated by
\begin{displaymath}
\{m\in M_{\gamma+p}| x_{0}^{\gamma}x_{r-1}^{p}\leq_{\textup{lex}} m \}.
\end{displaymath}
For all $d\geq \gamma+p$, we can derive
\begin{displaymath}
I_{d}=\left[\bigoplus_{i=0}^{r-2} x_{0}^{\gamma}x_{i}k[x_{i}, x_{i+1}, \ldots, x_{r}]_{d-\gamma-1}\right] \oplus x_{0}^{\gamma}x_{r-1}^{p}k[x_{r-1}, x_{r}]_{d-p-\gamma}
\end{displaymath}
so that
\begin{displaymath}
Q_{P}(d)=\dim_{k}I_{d}=\sum_{i=0}^{r-2}\binom{r-i+d-\gamma-1}{r-i}+\binom{r-(r-1)+d-p-\gamma}{1}.
\end{displaymath}
This implies that $g_{P}=p+\gamma$ by \cite[(2.9)]{Gotzmann}. On the other hand,
\begin{displaymath}
\begin{split}
P(d)&=\dim_{k}(S/I)_{d}\\
&=\left\vert\{m\in M_{d}\vert m<_{\textup{lex}}x_{0}^{\gamma}x_{r-1}^{p}x_{r}^{d-\gamma-p}  \}\right\vert\\
&=\left\vert\{m\in M_{d}\vert x_{0}^{\gamma}x_{r}^{d-\gamma}\leq_{\textup{lex}} m<_{\textup{lex}}x_{0}^{\gamma}x_{r-1}^{p}x_{r}^{d-\gamma-p}  \}\right\vert\\
&\quad + \left\vert\{m\in M_{d}\vert m<_{\textup{lex}}x_{0}^{\gamma}x_{r}^{d-\gamma}  \}\right\vert\\
&=p+\binom{r+d}{r}-\binom{r+d-\gamma}{r},
\end{split}
\end{displaymath}
for all $d\geq g_{P}$. Therefore, $g_{P}=\gamma+p$ if \eqref{goodsit} holds. If $g_{P}=l_{P}+\gamma+1$ and $r\geq 2$, then one of the following  statements is true.
\begin{itemize}
\item $r=2$ and $\gamma\neq 0$.
\item $p\leq 2$.
\end{itemize}
In other words, there is no general positive relationship between the Castelnuovo-Mumford regularity and the Kempf index of a Hilbert point. However, we can still expect that there is a positive relationship between the Kempf index and the Castelnuovo-Mumford regularity of 1-dimensional projective scheme by the first bullet above.
\bibliographystyle{acm}
\bibliography{gotzmann}

\begin{thebibliography}{10}

\bibitem{Bayer}
{\sc Bayer, D., and Morrison, I.}
\newblock Standard bases and geometric invariant theory. {I}. {I}nitial ideals
  and state polytopes.
\newblock {\em J. Symbolic Comput. 6}, 2-3 (1988), 209--217.
\newblock Computational aspects of commutative algebra.

\bibitem{VGIT}
{\sc Dolgachev, I.~V., and Hu, Y.}
\newblock Variation of geometric invariant theory quotients.
\newblock {\em Inst. Hautes \'Etudes Sci. Publ. Math.}, 87 (1998), 5--56.
\newblock With an appendix by Nicolas Ressayre.

\bibitem{Donaldson}
{\sc Donaldson, S.~K.}
\newblock Scalar curvature and stability of toric varieties.
\newblock {\em J. Differential Geom. 62}, 2 (2002), 289--349.

\bibitem{Eisenbudf}
{\sc Eisenbud, D.}
\newblock {\em Commutative algebra}, vol.~150 of {\em Graduate Texts in
  Mathematics}.
\newblock Springer-Verlag, New York, 1995.
\newblock With a view toward algebraic geometry.

\bibitem{Eisenbud}
{\sc Eisenbud, D.}
\newblock {\em The geometry of syzygies}, vol.~229 of {\em Graduate Texts in
  Mathematics}.
\newblock Springer-Verlag, New York, 2005.
\newblock A second course in commutative algebra and algebraic geometry.

\bibitem{Elias}
{\sc Elias, J., Giral, J.~M., Mir{\'o}-Roig, R.~M., and Zarzuela, S.}, Eds.
\newblock {\em Six lectures on commutative algebra}.
\newblock Modern Birkh\"auser Classics. Birkh\"auser Verlag, Basel, 2010.
\newblock Papers from the Summer School on Commutative Algebra held in
  Bellaterra, July 16--26, 1996, Reprint of the 1998 edition.

\bibitem{Gotzmann}
{\sc Gotzmann, G.}
\newblock Eine {B}edingung f\"ur die {F}lachheit und das {H}ilbertpolynom eines
  graduierten {R}inges.
\newblock {\em Math. Z. 158}, 1 (1978), 61--70.

\bibitem{Hyeon}
{\sc Hassett, B., and Hyeon, D.}
\newblock Log minimal model program for the moduli space of stable curves: the
  first flip.
\newblock {\em Ann. of Math. (2) 177}, 3 (2013), 911--968.

\bibitem{Hesselink}
{\sc Hesselink, W.~H.}
\newblock Uniform instability in reductive groups.
\newblock {\em J. Reine Angew. Math. 303/304\/} (1978), 74--96.

\bibitem{Hoskins}
{\sc Hoskins, V.}
\newblock Stratifications associated to reductive group actions on affine
  spaces.
\newblock {\em Q. J. Math. 65}, 3 (2014), 1011--1047.

\bibitem{Kempf}
{\sc Kempf, G.~R.}
\newblock Instability in invariant theory.
\newblock {\em Ann. of Math. (2) 108}, 2 (1978), 299--316.

\bibitem{M3}
{\sc Lee, C.}
\newblock Instability and singularity of projective hypersurfaces.
\newblock {\em Proc. Amer. Math. Soc. 146}, 12 (2018), 5015--5023.

\bibitem{Ian}
{\sc Morrison, I., and Swinarski, D.}
\newblock Gr\"obner techniques for low-degree {H}ilbert stability.
\newblock {\em Exp. Math. 20}, 1 (2011), 34--56.

\bibitem{Mukai}
{\sc Mukai, S.}
\newblock {\em An introduction to invariants and moduli}, vol.~81 of {\em
  Cambridge Studies in Advanced Mathematics}.
\newblock Cambridge University Press, Cambridge, 2003.
\newblock Translated from the 1998 and 2000 Japanese editions by W. M. Oxbury.

\bibitem{GIT}
{\sc Mumford, D., Fogarty, J., and Kirwan, F.}
\newblock {\em Geometric invariant theory}, third~ed., vol.~34 of {\em
  Ergebnisse der Mathematik und ihrer Grenzgebiete (2) [Results in Mathematics
  and Related Areas (2)]}.
\newblock Springer-Verlag, Berlin, 1994.

\bibitem{Murai}
{\sc Murai, S.}
\newblock Gotzmann monomial ideals.
\newblock {\em Illinois J. Math. 51}, 3 (2007), 843--852.

\bibitem{GTian}
{\sc {Paul}, S.~T., and {Tian}, G.}
\newblock {CM Stability and the Generalized Futaki Invariant I}.
\newblock {\em ArXiv Mathematics e-prints\/} (May 2006).

\end{thebibliography}
\end{document}